\documentclass[11pt]{amsart}

\usepackage[hmargin=1.25in,vmargin=1.25in]{geometry}

\usepackage{amsmath,amssymb,amsthm}
\usepackage{mathtools}

\usepackage{graphicx}
\usepackage{subcaption}
\usepackage{float}
     \usepackage{tikz}
\usetikzlibrary{quotes,angles}

\usepackage{array}
\usepackage{tensor}
\usepackage{paralist}
\usepackage[english]{babel}
\usepackage{xcolor}
\usepackage{hyperref}
\usepackage{microtype}

\theoremstyle{plain}
\newtheorem{theo}{Theorem}[section]
\newtheorem{prop}[theo]{Proposition}
\newtheorem{lem}[theo]{Lemma}

\newtheorem{cor}[theo]{Corollary}
\theoremstyle{definition}
\newtheorem{def1}[theo]{Definition}

\theoremstyle{remark}

\renewcommand{\Im}{\operatorname{Im}}
\newcommand{\vol}{\operatorname{Vol}}

\newcommand{\Id}{\operatorname{Id}}

\newcommand{\R}{\mathbb{R}}

\newcommand{\N}{\mathbb{N}}

\renewcommand{\phi}{\varphi}
\renewcommand{\theta}{\vartheta}
\renewcommand{\epsilon}{\varepsilon}

\newcommand{\loc}{\mathrm{loc}}

\newcommand{\rmi}{\mathrm{i}}            
\newcommand{\der}{\mathrm{d}}            
\newcommand{\comp}{\mathrm{comp}}        

\newenvironment{rem}{\medskip\noindent{\it Remark:\/} }{\medskip}

\mathtoolsset{showonlyrefs=true}
\title[Scattering matrix for the $p$-form Laplacian]{The scattering matrix for the $p$-form Laplacian on asymptotically conic manifolds}
\date{}

\author{Nelia Charalambous}
\address{University of Cyprus} \email{charalambous.nelia@ucy.ac.cy}

\author{Alden Waters}
\address{Leibniz Universit\"at Hannover} \email{alden.waters@math.uni-hannover.de}
\thanks{Corresponding author: Alden Waters (alden.waters@math.uni-hannover.de).}



\begin{document}
	\maketitle
	
\begin{abstract}
The Hodge decomposition at infinity gives rise to distinct tangential scattering channels for differential forms on asymptotically conic manifolds. We construct co-closed generalized eigenforms for the Hodge Laplacian on $p$-forms in dimension $n\geq 3$ and analyze the resulting Bessel and Hankel systems. From these we determine the scattering matrix modulo smoothing operators, identifying it as a Fourier integral operator associated to the time $\pi$ geodesic flow and 
whose principal symbol encodes fibre-wise parallel transport with respect to the Levi-Civita connection on differential forms. In this setting this lifts the main result of Melrose and Zworski in \cite{MZInvent} to  differential forms and identifies additional scattering phase. The construction also yields a spectral representation in the co-closed sector, including a Stone-type formula and a functional calculus. In dimension $3$ and degree $p=1$, this framework describes divergence-free fields and thus the electric-field component of Maxwell's equations.
\end{abstract}

\section{Introduction}

{
One of the fundamental geometric results in scattering theory on asymptotically conic manifolds is the description {by Melrose and Zworski \cite{MZInvent}} of the scalar scattering matrix as a Fourier integral operator associated with geodesic flow {on the boundary at infinity} at time $\pi$. Thus the leading singular behaviour of the scattering matrix is governed by the intrinsic geometry of the cross-section at infinity. The purpose of this paper is to make an analogous description explicit {for the Hodge Laplacian on} co-closed differential forms, under the asymptotic hypotheses stated below.

The passage from functions to $p$-forms introduces bundle-valued boundary data. The Hodge half-wave group transports these data by Levi-Civita parallel transport along geodesics. At the same time{, the exterior derivative $d$, its adjoint $\delta$,} and the co-closedness constraint require us to resolve the tangential and normal components at infinity {and the Hodge decomposition on the} cross-section. Our main result identifies the tangential scattering matrix in these components and makes its geometric content explicit{.

Let $(M,g)$ be an asymptotically conic manifold of dimension $n\geq 3$. Near infinity, the end is identified with $(0,\epsilon_1)\times N$ and the metric has the form
\[
        g=\frac{dx^2}{x^4}+\frac{h(x)}{x^2},
        \qquad h(0)=h_0,
\]
where $(N,h_0)$ is a} closed {Riemannian manifold.} Under the infinite-order hypothesis $h(x)-h_0=O(x^\infty)$, Theorem~\ref{lengths} identifies {the phase-normalized tangential scattering matrix}, at each fixed $\lambda>0$, modulo smoothing operators. On {the} nonharmonic {co-closed and closed boundary components,} the two model factors are
{\[
-i\exp\left(-i\pi\sqrt{\left(\frac n2-p-1\right)^2+\Delta_N}\right)
\quad \text{and} \quad
i\exp\left(-i\pi\sqrt{\left(\frac{n}{2}-p\right)^2+\Delta_N}\right),
\]}
respectively, where $\Delta_N$ is {the Hodge Laplacian on $p$-forms} on $N$. The harmonic angular case $\sigma=0$ is treated separately in \eqref{harmonic-phase}, also for $\lambda>0$. The exact-cone model and its Poisson operator use the same relative realization {of the Hodge Laplacian.} The harmonic forms are retained in the distinct tangential and normal {components of the form decomposition.}

Under the infinite-order metric hypothesis, Corollary~\ref{cor:fio} shows that the tangential scattering matrix is an order-zero Fourier integral operator associated with the graph of the time-$\pi$ geodesic flow of $(N,h_0)$. Its principal symbol combines Levi-Civita parallel transport with the Hodge projections and the phases $-i$ and $i$ on the co-closed and closed components, respectively. The constant shifts under the square roots affect only lower-order terms; the harmonic contribution is smoothing.

The contribution is the identification of the co-closed $p$-form scattering matrix with these explicit boundary operators, and hence with this geometric Fourier integral structure. The scalar theorem of Melrose and Zworski holds under more general metric hypotheses. An explicit scalar {model calculation appears in Christiansen \cite{christiansen}}, and {related inverse and asymptotic questions were studied in \cite{JSB,LOSZ}.} Scattering {for differential forms on conical ends was studied by Parnovski \cite{Parnovski}}; related spectral approaches to differential-form scattering appear in \cite{OS,RBM}{. We use the scattering-metric and resolvent framework of \cite{melrosebook,MZInvent,HV,HVLaplace,HW,gh1,gh2,gh3}, together with the theory of the Hodge Laplacian on non-compact manifolds \cite{GS,CCH,Bueler,gaffney,donnelly}. For background on resonances we refer to \cite{DZbook,PZ}, and for the functional-analytic setting to \cite{RS}.}   

For the Hodge Laplacian, Bei, G\"uneysu and M\"uller \cite{BGM} prove existence and completeness of wave operators under conformal metric perturbations. In the scalar setting, Baskin and Yang \cite{BY} construct the resolvent and scattering matrix on cones and truncated cones, while Yang \cite{Yang} relates the scattering matrix on product cones to the diffraction coefficient of the half-wave kernel.  

Our proof uses {the Hodge decomposition on $N$} to reduce the model eigenform equations to radial systems, solves them {in terms of Bessel and Hankel functions}, and constructs co-closed generalized eigenforms. The resulting model scattering factors are compared with those on $M$ through resolvent and scattering-amplitude arguments. Under the weaker assumption $h(x)-h_0=O(x^{n_0})$, $n_0\geq3$, the comparison remainder has the mapping property $H^{-s}(N;\Lambda^pT^*N)\to L^2(N;\Lambda^pT^*N)$ for $s<n_0-1$ stated in Theorem~\ref{lengths}; the modulo-smoothing Fourier integral conclusion above is asserted under infinite-order agreement{.

A second part of the paper develops the} associated spectral representation, following the strategy of \cite{OS,RBM} while retaining the differential-form constraints. This includes a Stone-type formula under the additional dimensional assumptions of Theorem~\ref{stone1}, an associated functional calculus, {and a kernel expansion for rapidly decaying functions of the Hodge Laplacian. In dimension $n=3$ with $p=1$, co-closed one-forms correspond to divergence-free vector fields and} give the electric-field scattering problem for Maxwell's equations. This case is included in the main scattering theorem; the additional restrictions in Theorem~\ref{stone1} concern the later spectral statement{. In the Euclidean model}, the generalized eigenforms agree with the Maxwell {representations used in \cite{YFW}. Related low-energy} results for {cylindrical ends appear in \cite{mueller}.}
\par}

\section{Methods}
\subsection{The Laplacian on forms on the exact manifold $\widetilde M_0$}

In this section, we analyze the Laplacian on forms for the exact manifold $\widetilde M_0$  with metric $g_0=x^{-4} dx^2 + x^{-2} h_0$.
{On the cone, $\Delta_p$ denotes the relative Hodge Laplacian \cite{Vertman} associated with
\[
D=d_{\min},\qquad D^*=\delta_{\max},\qquad
\Delta_p=D_p^*D_p+D_{p-1}D_{p-1}^*.
\]
Here $d_{\min}$ is the graph closure of $d$ on compactly supported smooth forms. All model Poisson operators use this same realization.\par}

Note that the change of variables $x=1/r$ transforms the end into a warped product with metric $g=dr^2 + r^2 h_0$. Following Donnelly~\cite{donnelly}, the Hodge decomposition on $N$ allows us to write any $L^2$-integrable $p$-form on $(\widetilde M_0, g)$ as
\begin{equation} \label{decomp}
\psi= \omega_1 \oplus dx\wedge \omega_2 \oplus( \tilde{\omega}_1 + dx\wedge \tilde{\omega}_2),
\end{equation}
where $\omega_1$ (resp. $\tilde{\omega}_1$) is a co-closed (resp. closed) $p$-form on $N$ parametrized by $x$, and $\omega_2$ (resp. $\tilde{\omega}_2$) is a closed (resp. co-closed) $(p-1)$-form on $N$ parametrized by $x$.

Let $x^2\partial_1 = x^2\partial_{x_1} =  x^2\partial_{x}$ and $x\partial_j = x\partial_{x_j}$ for $j>1$ form an orthonormal frame of $T\widetilde{M}_0$, and let $x^{-2}dx^1$ and $x^{-1}dx^j$ for $j>1$ form the dual coframe of $T^*\widetilde{M}_0$. We write $d$ and $d_N$ for the exterior derivatives on $M$ and $N$, respectively, and $\delta$, $\delta_N$ for their $L^2$-adjoints  on $(M,g)$ and $(N,h_0)$, respectively.

Following \eqref{decomp}, via the Hodge decomposition on $N$, any form in $L^2(\widetilde M_0 ; \Lambda^p T^*\widetilde{M}_0)$ can be written as a linear combination of forms of the following types:
\begin{align}
&f(x)\eta_1& \label{type1}\\
&f(x)\,dx\wedge \eta_2& \label{type2}\\
&\frac{1}{\sqrt{\sigma}}f_1(x)d_N\eta_3+f_2(x)\, dx\wedge \eta_3,  \label{type3}
\end{align}
where $\eta_1,\eta_2,\eta_3$ are normalized eigenforms on $N$ with eigenvalue $\sigma$, with $\eta_2$ closed and $\eta_1,\eta_3$ co-closed. As Cheeger~\cite{cheegerheat} points out, the Laplacian also leaves invariant $p$-forms of the type
\begin{equation}
f_1  \eta_4+  f_2  \tfrac{1}{\sqrt{\sigma}} \, dx\wedge \delta_N \eta_4  \label{type4}
\end{equation}
where $\eta_4$ is a closed eigenform on $N$ with eigenvalue $\sigma$.

We examine the action of the Laplacian on each of these types of forms, referred to as forms of the {\it first, second, third} and {\it fourth} types, respectively. Harmonic eigenforms are included among forms of the first and second types. Forms of the first and second types are dual with respect to the Hodge star operator, and forms of the third and fourth types are also dual. The action on forms of the first and second type is computed in the following lemma.
\begin{lem} \label{lem1}
Let $'$ denote differentiation with respect to $x$. We define the operators $\Delta_1$ and $\Delta_2$ on generic $C^2(0,\epsilon_1)$ functions of $x$, $f(x)$ as follows:
\begin{equation}
\Delta_1 f=  - x^4 f''  + (n-2p-3) x^3  f'   + \sigma \, x^2 f
\end{equation}
and
\begin{equation}
\Delta_2 f=    -x^4 f''+(n-2p-5)x^3 f' + (3(n-2p-1)+\sigma) \, x^2 f.
\end{equation}
Let $\eta_1$ be a co-closed $p$-eigenform on $N$ with eigenvalue $\sigma$, and set $\psi_1=f(x)\,\eta_1$ as above. Then, the action of the Laplacian on $\psi_1$ is equivalent to the action of the operator $\Delta_1$:
\[
\Delta (f \eta_1) = (\Delta_1 f)  \eta_1.
\]
Let $\eta_2$ be a closed $(p-1)$-eigenform with eigenvalue $\sigma$, and set $\psi_2=f(x)\,dx\wedge \eta_2$. Then the action of the Laplacian on $\psi_2$ is equivalent to the action of the operator $\Delta_2$:
 \[
\Delta (f \,  dx\wedge\eta_2) =  (\Delta_2 f) \,   dx\wedge\eta_2.
\]
 
\end{lem}

\begin{proof}
For $f\in C^2(0,\epsilon_1)$ and any $p$-form $\eta_1$ on $N$,
\begin{equation*}
\begin{split}
d\psi_1 & =d(f \eta_1)=f'  \,dx\wedge\eta_1+f \,d_N\eta_1  \\
\delta\psi_1 &=\delta(f \eta_1)= x^2 f \,\delta_N\eta_1
\end{split}
\end{equation*}
which yields
\begin{equation*}
\begin{split}
\Delta \psi_1
= \Delta(f \,\eta_1)=&  \left[ - x^4 f''  + (n-2p-3) x^3  f' \right]\,\eta_1 +  x^2 f \, \Delta_N\eta_1\\
& +   2x f  \, dx \wedge  \delta_N \eta_1 .
\end{split}
\end{equation*}
Whenever $\eta_1$ is a co-closed $p$-eigenform with $\delta_N\eta_1=0$ and eigenvalue $\sigma$,
\[
\Delta (f \eta_1) = (\Delta_1 f)  \eta_1.
\]
Similarly, for a $(p-1)$-eigenform $\eta_2$ on $N$,
\begin{equation*}
\begin{split}
d\psi_2 & =d(f\, dx\wedge\eta_2)= - f  \,dx\wedge d_N\eta_2  \\
\delta\psi_2 &=\delta(f  \, dx\wedge \eta_2)= \left[-x^4 f'  +(n-2p-1)\, x^3 f  \right] \eta_2 - x^2 f  \, dx \wedge \delta_N \eta_2.
\end{split}
\end{equation*}
Hence
\begin{equation*}
\begin{split}
\Delta  \psi_2   =  \Delta(f \,dx\wedge\eta_2) =&  \left[ -x^4 f''+(n-2p-5)x^3 f' + 3(n-2p-1) x^2 f\right] \, dx \wedge  \eta_2\\
& + x^2 f \, dx\wedge\Delta_N \eta_2 + 2 x^3 f  d_N\eta_2.
\end{split}
\end{equation*}
Therefore, whenever $\eta_2$ is a  closed $(p-1)$-eigenform on $N$ with $d_N\eta_2=0$ and eigenvalue $\sigma$,  we get
\[
\Delta (f \,  dx\wedge\eta_2) =  (\Delta_2 f) \,   dx\wedge\eta_2. \qedhere
\]
\end{proof}

We next consider the action of the Laplacian on forms of the third type. 
\begin{prop}\label{prop1}
The action of the Laplacian on forms of the third type where $\eta_3$ is a co-closed $(p-1)$-eigenform with eigenvalue $\sigma\neq 0$ is equivalent to the differential operator
\[
\Delta_3\left[
          \begin{array}{c}
            f_1 \\
            f_2 \\
          \end{array}
        \right] =
        \left[
          \begin{array}{c}
           \Delta_1 f_1 +  2 \sqrt{\sigma} \,  x^3 f_2  \\
            \Delta_2 f_2+ 2 \sqrt{\sigma}\, x \, f_1  \\
          \end{array}
        \right]
\]
where $\Delta_1$ and $\Delta_2$ are defined in Lemma \ref{lem1}, on generic $C^2(0,\epsilon_1)$ functions of $x$, $f_1$ and $f_2$ respectively.
\end{prop}
\begin{proof}

Taking a $(p-1)$-eigenform $\eta_3$ with $\delta_N\eta_3=0$ and eigenvalue $\sigma$,
\begin{equation*}
  \begin{split}
\Delta (f_1 d_N\eta_3+ f_2 \, dx\wedge \eta_3) =   \left[\Delta_1 f_1 + 2 x^3 f_2 \right] d_N\eta_3  + \left[\Delta_2 f_2 + 2\sigma \, x \,f_1   \right] dx \wedge \eta_3
  \end{split}
\end{equation*}
since $\delta_N d_N\eta_3=\sigma \eta_3$. This defines an operator on the unnormalized pair $(f_1, f_2)$:
\[
\Delta_3\left[
          \begin{array}{c}
            f_1 \\
            f_2 \\
          \end{array}
        \right] =
        \left[
          \begin{array}{c}
            \Delta_1 f_1  + 2 x^3 f_2\\
            \Delta_2 f_2 + 2 \sigma \,  x \, f_1.  \\
          \end{array}
        \right].
\]
Replacing $(f_1, f_2)$ by $(\tfrac{1}{\sqrt\sigma}f_1, f_2)$ to match the normalized form \eqref{type3} recovers the $\sqrt{\sigma}$ coefficients in the statement.
Note that for forms of the third type, we could have equivalently considered a $p$-eigenform $\eta_4$ with $d_N\eta_4=0$ such that
\begin{equation*}
  \begin{split}
\Delta (f_1  \eta_4+ f_2 \, dx\wedge \delta_N \eta_4) =& \left[\Delta_1 f_1 +  2 \sigma  \, x^3  f_2   \right]  \eta_4  + \left[\Delta_2 f_2   + 2x f_1  \right] dx \wedge \delta_N \eta_4
\end{split}
\end{equation*}
since in this case $d_N\delta_N \eta_4=\sigma \eta_4$, leading to the operator $\Delta_4$:
\[
\Delta_4\left[
          \begin{array}{c}
            f_1 \\
            f_2 \\
          \end{array}
        \right] =
        \left[
          \begin{array}{c}
           \Delta_1 f_1  + 2 \sigma  \, x^3f_2  \\
            \Delta_2 f_2  + 2x \,f_1  \\
          \end{array}
        \right].
\]

For  $\sigma\neq 0$ we can show that the two operators $\Delta_3, \Delta_4$ are equivalent after observing that for $\eta_3$ and $\eta_4$ as above,
\begin{equation*}
  \begin{split}
 \Delta (  f_1 \tfrac{1}{\sqrt{\sigma}} \, d_N\eta_3+ f_2 \, dx\wedge \eta_3)
       & = \left[\Delta_1 f_1 + 2 \sqrt{\sigma} \,  x^3 f_2 \right] \tfrac{1}{\sqrt{\sigma}} \, d_N\eta_3  + \left[\Delta_2 f_2 + 2 \sqrt{\sigma} \,x \, f_1 \right] \, dx \wedge \eta_3
  \end{split}
\end{equation*}
and
\begin{equation*}
  \begin{split}
  \Delta (f_1  \eta_4+  f_2  \tfrac{1}{\sqrt{\sigma}} \, dx\wedge \delta_N \eta_4)
     =& \left[\Delta_1 f_1  + 2 \sqrt{\sigma}\,  x^3 f_2 \right]  \eta_4 + \left[\Delta_2 f_2 + 2 \sqrt{\sigma}\, x \, f_1 \right] \tfrac{1}{\sqrt{\sigma}} \, dx \wedge \delta_N \eta_4.
 \end{split}
\end{equation*}
It suffices to work with $\Delta_3$ alone.
\end{proof}

When $\sigma=0$, the off-diagonal coupling in $\Delta_3$ vanishes and the two components decouple, reducing to the operators $\Delta_1$ and $\Delta_2$ studied in Lemma~\ref{lem1}.

\subsection{Model eigenforms on the exact manifold $\widetilde M_0$}
The Hodge decomposition allows us to solve the eigenvalue problem for the Laplacian on the  manifold $\widetilde M_0$  which has exact metric everywhere. This amounts to finding the eigenforms for each of the operators $\Delta_1, \Delta_2, \Delta_3$.  For $\Delta_3$, it is more convenient to work with the rescaled operator $\tilde{\Delta}_3$ in \eqref{eq7}.
We use the following rescaling:
\begin{equation} \label{rescal}
\begin{split}
   &  f_1= x^\alpha H_1 \quad \text{with} \quad \alpha=\frac{n-2p-1}{2} \\
   & f_2= x^\beta H_2 \quad \text{with} \quad \beta=\frac{n-2p-3}{2}.
\end{split}
\end{equation}
This substitution yields the operators:
\begin{equation} \label{eq5}
\begin{split}
&\tilde{\Delta}_1H_1= -  x^4 H_1''  -2 \, x^3 \,H_1'    + (\alpha(\alpha-1)+\sigma) \, x^2 \,H_1 \\
&\tilde{\Delta}_2 H_2 = -  x^4 H_2''    -2 \, x^3\, H_2'   + ((\beta+2)(\beta+3)+\sigma) \, x^2  \,  H_2\\
&\tilde{\Delta}_3\left[
          \begin{array}{c}
            H_1 \\
            H_2 \\
          \end{array}
        \right] =
        \left[
          \begin{array}{c}
           \tilde{ \Delta}_1 H_1 +2 \sqrt{\sigma}\, x^2 \, H_2  \\
            \tilde{\Delta}_2 H_2 +2\sqrt{\sigma}\, x^2 \, H_1 \\
          \end{array}
        \right].
\end{split}
\end{equation}	
These operators are compatible with the rescaling \eqref{rescal} in the sense that $\tilde\Delta_i$ acts on the rescaled function $H_i$ rather than on $f_i$, as verified by
\begin{equation} \label{eq6}
\begin{split}
&\Delta_1(x^\alpha {H})= x^\alpha\left[ -  x^4  {H}''  -2 \, x^3 \, {H}'    + (\alpha(\alpha-1)+\sigma) \, x^2 \, {H} \right] \\
&\Delta_2 (x^\beta{H})  = x^\beta \left[  -  x^4  {H}''    -2 \, x^3\, {H}'   + ((\beta+2)(\beta+3)+\sigma) \, x^2  \,  {H} \right]\\
&\Delta_3\left[
          \begin{array}{c}
            x^\alpha {H}_1 \\ \\
            x^\beta {H}_2 \\
          \end{array}
        \right] =
        \left[
          \begin{array}{c}
          x^\alpha\left( \tilde{ \Delta}_1 {H}_1 +2 \sqrt{\sigma}\, x^2 \,  {H}_2 \right) \\
        \\
           x^\beta \left( \tilde{\Delta}_2  {H}_2 +2\sqrt{\sigma}\, x^2 \,  {H}_1 \right) \\
          \end{array}
        \right].
\end{split}
\end{equation}	
The eigenvalue problem $\tilde\Delta_3[H_1,H_2]^T = \lambda^2[H_1,H_2]^T$ for $\lambda^2\geq 0$ therefore takes the form
\begin{equation}\label{eq7}
\begin{split}
&\tilde{\Delta}_3\left[
          \begin{array}{c}
           H_1 \\
            H_2 \\
          \end{array}
        \right] =
        \left[
          \begin{array}{c}
           \tilde{ \Delta}_1 H_1 +2 \sqrt{\sigma}\, x^2 \,  H_2   \\
         \tilde{\Delta}_2  H_2 +2\sqrt{\sigma}\, x^2 \,  H_1  \\
          \end{array}
        \right] =
        \left[
          \begin{array}{c}
          \lambda^{2} H_1 \\
          \lambda^{2}   H_2 \\
          \end{array}
        \right].
\end{split}
\end{equation}	
We now solve the eigenform problem. A similar decomposition also appears in \cite{cheegerheat}, Section~3, but explicit formulas are needed to systematically develop the spectral calculus. 
\begin{prop} \label{propsol}
For a fixed $\sigma{>0}$ define $\ell_p=  \sqrt{\left(\frac{n}{2}-p-1\right)^2+\sigma}$, and let 
\begin{align} \label{eigenform-type1}
f_{1,p}(x,\lambda)=x^{\frac{n-2p-2}{2}}J_{\ell_p}\left(\frac{\lambda}{x}\right), 
\end{align}
and
\begin{equation}
\begin{split} \label{type2sol}
f_{2,p}(x,\lambda)&=x^{\frac{n-2p-4}{2}}J_{\ell_{p-2}}\left(\frac{\lambda}{x}\right).
\end{split}
\end{equation}
Then, $\Delta_i(f_i)=\lambda^{2} f_i$ for $i=1,2$ and $\lambda^2\in\mathbb{R}_0^+$.
In particular, if $\eta_1$ is a normalized co-closed $p$-eigenform on $N$ with eigenvalue $\sigma$ and $\eta_2$ is a normalized closed $(p-1)$-eigenform with the same eigenvalue, then
\[
\omega_{1,p} = f_{1,p}(x,\lambda)\eta_1, \quad \omega_{2,p} =  f_{2,p}(x,\lambda)\,dx\wedge \eta_2
\]
are $p$-eigenforms on $\widetilde M_0$ of the first and second type respectively, with eigenvalue $\lambda^2$. 

In addition $\omega_{1,p}$ is co-closed and $\omega_{2,p}$ is closed on $\widetilde M_0$. When $\xi$ is a closed $p$-eigenform on $N$ and $F(x,\lambda)\,dx\wedge\xi$ is a degree-$(p+1)$ eigenform of the second type, then $\delta(F(x,\lambda)\,dx\wedge \xi)$ is a co-closed $p$-eigenform of the third type.  
\end{prop}
\begin{proof}
That $f_{1,p}$ and $f_{2,p}$ satisfy the respective differential equations follows directly from \eqref{parambess} in the Appendix. Observe that the function  $H_{1,p} = x^{-\frac{1}{2}}J_{\ell_p}\left(\frac{\lambda}{x}\right)$ is a solution to the eigenvalue problem 
\[
\tilde{\Delta}_1 H =\lambda^2 H
\]
for $\Delta_1$ and $\alpha$ as defined in \eqref{rescal}, \eqref{eq5}. Moreover, $H_{2,p} = x^{-\frac{1}{2}}J_{\ell_{p-2}}\left(\frac{\lambda}{x}\right)$ is a solution to 
\[
\tilde{\Delta}_2 H =\lambda^2 H
\]
for $\Delta_2$ and $\beta$ as in \eqref{rescal}, \eqref{eq5}. That $\omega_{1,p}$ is co-closed and $\omega_{2,p}$ is closed follows immediately from the proof of Lemma~\ref{lem1}.

From \cite{cheegerheat}, Section~3, page~587, $\delta$ maps forms of the second type to forms of the third type, and $d$ maps forms of the first type to forms of the fourth type (defined in Prop. \ref{prop1}). More precisely, if $\xi$ is a closed $p$-eigenform on $N$ with eigenvalue $\sigma$ and $F(x,\lambda)$ is such that $F(x,\lambda)\,dx\wedge \xi$ is a degree-$(p+1)$ second-type form with eigenvalue $\lambda^2$, then $\delta(F(x,\lambda)\,dx\wedge \xi)$ is a co-closed $p$-form eigenfunction with eigenvalue $\lambda^2$: co-closedness follows from $\delta^2=0$, and the eigenvalue property from the fact that $\delta$ commutes with the Hodge Laplacian. Notice that a choice of $(p+1)$-form then changes the index of the solution to $\ell_{p-1}(\sigma)$ as well as the factor out front e.g. $$F(x,\lambda)=x^{\frac{n-2p-6}{2}}J_{\ell_{p-1}(\sigma)}\left(\frac{\lambda}{x}\right).$$ This is what we will use later to construct the generalized eigenfunctions. 
\end{proof}

{For harmonic angular data and $\lambda>0$, the tangential solution is unchanged. In the normal auxiliary $(p+1)$-form, when $n=2p+1$, the relative conditions instead select
\[
u(r)=\sqrt r J_{1/2}(\lambda r),\qquad
v(r)=\sqrt r J_{-1/2}(\lambda r),
\]
with $u(0)=0$ and $v'(0)=0$, respectively where $r=1/x$. Applying $\delta$ to the normal solution gives the tangential relative solution. For other degrees the original Bessel orders apply.\par}

\subsection{Generalized eigenfunctions}

In this section, we construct generalized eigenfunctions for the Helmholtz problem on the exact manifold $\widetilde M_0$ and then on the general manifold $M$, following the spectral calculus developed in \cite{OS}. We let $\langle\cdot, \cdot\rangle_{\widetilde{M}_0}$ be the inner product on $L^2(\widetilde{M}_0; \Lambda^pT^{*}\widetilde{M}_0)\oplus L^2(\widetilde{M}_0; \Lambda^{p-1}T^{*}\widetilde{M}_0)$.
\begin{lem} \label{basis}
For each eigenvalue $\sigma$ of the Laplacian on $N$ there exists a basis of co-closed and closed $p$-eigenforms for $L^2(N;\Lambda^p T^*N)$ which we denote by $\{\phi_{\sigma}^{(p)}\}$ and $\{ \xi_{\sigma}^{(p)}\}$ respectively. The collection $\{\phi_{\sigma}^{(p)}, \xi_{\sigma}^{(p)}\}_p$ spans $L^2(N;\Lambda^p T^*N)$. Moreover, for $\sigma\neq 0$ we can arrange these bases  so that for any element $\xi_{\sigma}^{(p)}$ in the basis, we have
\begin{align}  \label{eval2}
\delta_N\xi_{\sigma}^{(p)}=\sqrt{\sigma}\phi_{\sigma}^{(p-1)}.
\end{align}
For $\sigma=0$, the harmonic forms on $N$ are simultaneously closed and co-closed; they both are retained in the construction of the generalized eigenfunctions as they correspond to distinct components of the form decomposition.
\end{lem} 
\begin{proof}
For each eigenvalue $\sigma$ of $\Delta_N$ on $p$-forms over $N$, the eigenspace decomposes into co-closed and exact pieces: there exist either a co-closed $p$-form $\eta_1\in\ker(\delta_N)\subset\Lambda^pT^*N$ with $\Delta_N\eta_1=\sigma\eta_1$, or a closed $p$-form $d\eta_2$ with $\eta_2\in\Lambda^{p-1}T^*N$ and $\Delta_N(d\eta_2)=\sigma\,d\eta_2$. An orthonormal basis for $L^2(N;\Lambda^pT^*N)$ therefore consists of such co-closed eigenforms $\{\phi_\sigma^{(p)}\}$ and closed eigenforms $\{\xi_\sigma^{(p)}\}$.

By Poincar\'e duality a $p$-form $\omega$ on $N$ is closed if and only if $\star^N \omega$ is a co-closed $(n-1-p)$-form. Moreover, the Hodge star operator commutes with the Laplacian, hence it preserves eigenforms. Using the fact that $\star^N  \star^N  \omega = (-1)^{p(n-1-p)} \omega$, we obtain that  $\phi^{(n-1-p)}_{\sigma}=C_o\star^N\xi_{\sigma}^{(p)}$ if and only if
\begin{align} \label{eval1}
\star^N\phi_{\sigma}^{(n-1-p)}=C_o(-1)^{p(n-1-p)}\xi_{\sigma}^{(p)},
\end{align} 
where $\star^N$ is the Hodge star operator on the compact manifold $N$, and $C_o$ is a constant. 

Finally, using the above, we obtain that whenever $\xi_{\sigma}^{(p)}$ is a closed $p$-eigenform, then $\delta_N \xi_{\sigma}^{(p)}$ is a co-closed $(p-1)$-eigenform with the same eigenvalue, hence we can rescale our bases for $\sigma\neq 0$ to get \eqref{eval2}. 
\end{proof} 

The proof of the main theorem works by conjugating the operators by the isometry $\Pi_p$, under whose image the computations are explicit. We start with $\delta$, $d$, and the Laplacian on $\widetilde{M}_0$. By the proof of Lemma \ref{lem1}, the operator $d$ can be written as
\[
d\psi = {\Pi}_{p+1}^{-1} \, A_p \, {\Pi}_p\psi
\]
where
\[ 
A_p= \begin{pmatrix} x \, d_N & 0 \\ x^2\partial_x +\left(\frac{n-1}{2}-p\right)\, x & -x \, d_N \end{pmatrix}.
\]
Similarly, the operator $\delta$ can be written as 
\[
\delta\psi = {\Pi}_{p-1}^{-1}  \, B_p \, {\Pi}_p\psi
\]
where
\begin{equation}\label{eq8}
B_p= \begin{pmatrix} x \, \delta_N &  - x^2\partial_x + \left(\frac{n-1}{2}-p +1 \right)\, x \\ 0 & -x \, \delta_N \end{pmatrix}.
\end{equation}
The Laplacian on $p$-forms $\psi$ is then given by
\[
\Delta \psi =  {\Pi}_{p}^{-1} \,\left( B_{p+1} A_p + A_{p-1} B_p \right) \, {\Pi}_p\psi
\]
where
\begin{equation}\label{eqn:Lapl}
\begin{split}
&B_{p+1} A_p + A_{p-1} B_p \\
&= \begin{pmatrix} - x^2\partial_x (x^2\partial_x) +x^2 \alpha (\alpha-1) + x^2 \Delta_N &   2x^2 \, d_N  \\  
2x^2 \, \delta_N &  - x^2\partial_x (x^2\partial_x) +x^2 (\beta+2) (\beta+3) + x^2 \Delta_N  \end{pmatrix}
\end{split}
\end{equation}
for $\alpha, \beta$ as in \eqref{rescal}. Note that we recover $\tilde{\Delta}_1$ and $\tilde{\Delta}_2$ from the diagonals. 

Lemma \ref{basis} and the Hodge decomposition results allow us to describe the sets of co-closed and closed $p$-forms over $\widetilde{M}_0$ using the basis of eigenforms over $N$. Recall that the Hilbert space $L^2(\widetilde{M}_0,\Lambda^pT^*\widetilde{M}_0)$ decomposes orthogonally into three invariant subspaces for $\Delta_p$:
\begin{align}
L^2(\widetilde{M}_0,\Lambda^pT^*\widetilde{M}_0)={\overline{\operatorname{ran}D_p^*}}\oplus \overline{dC_0^{\infty}(\widetilde{M}_0;\Lambda^{p-1}T^*\widetilde{M}_0)}\oplus \mathcal{H}^p(\widetilde{M}_0),
\end{align}
where $\mathcal{H}^p(\widetilde{M}_0)=\ker(\Delta_p)\cap L^2$ is the space of $L^2$ harmonic $p$-forms, and analogously for the manifolds $\widetilde{M}$ and $M$. 

Since the $*$ operator commutes with the Laplacian {as a differential expression}, and a $p$-form $\eta$ on $N$ is closed if and only if $*_N \eta$ is co-closed on $N$, we can use this duality to extract the eigenforms for $\Delta_2$ from the eigenforms of $\Delta_1$. {At the cone tip, $*$ interchanges the relative and absolute realizations.} By Hodge duality there is an analogue of the decomposition for closed $p$-forms on $\widetilde{M}_0$: in that case the $\xi_\sigma^{(p-1)}$ coefficients are unrestricted and the $\phi_\sigma^{(p)}$ coefficients vanish for $\sigma\neq 0$. This duality shows that it is enough to determine the Poisson operator and the scattering matrix only for co-closed forms on $\widetilde M$ or $M$. We solve the equations on an exact manifold $\widetilde{M}_0$ and then we can construct $\widetilde{M}$ by the usual gluing construction, gluing $\widetilde{M}=\widetilde{M_0}|_{Z}\cup K$ where $K$ is a compact set with a different metric c.f. \cite{GS} page 7. We will eventually find that to find the scattering matrix it then suffices to compare $\widetilde{M}$ to $\widetilde{M_0}|_Z$.

Let $H^{(i)}_\nu$ be the Hankel functions defined in \eqref{hankel} of the Appendix.  For a fixed $p$  define   $\ell_p(\sigma)= \sqrt{\left(\frac{n}{2}-p-1\right)^2+\sigma}$. 
{In the scattering construction, $\lambda>0$. The second-component sums below have $\sigma>0$; their individual harmonic extensions are $\tilde h^{(j)}_{2,\lambda}(\eta)=-\tilde h^{(j)}_{1,\lambda}(\eta)$ and $\tilde j_{2,\lambda}(\eta)=-\tilde j_{1,\lambda}(\eta)$ for $\eta\in\mathcal H^p(N)$. In combined operators and single orthonormal boundary bases, the harmonic coefficient belongs to the first component.}
\begin{def1} \label{def:hankel-solutions} For a $p$-form  $\Phi \in L^2(N;\Lambda^pT^*N)$, let  
\begin{align}
c_{1,\sigma} =\left(\Phi, \phi_{\sigma}^{(p)} \right)_{L^2(N;\Lambda^pT^*N)} \quad c_{2,\sigma} =\left(\Phi,  \xi_{\sigma}^{(p)} \right)_{L^2(N;\Lambda^{p}T^*N)}.
\end{align}
We define the outgoing and incoming Hankel-function solutions by
\begin{align*}
&\left(\frac{\pi\lambda}{2}\right)^{-\frac{1}{2}}\Pi_p\tilde{h}_{\lambda}^{(1)}(\Phi)(x,\theta)=\\& \nonumber
\sum\limits_{\sigma\geq 0}\frac{c_{1,\sigma}}{\sqrt{x}}(-i)^{\ell_{p}}H_{\ell_{p}}^{(1)}\left(\frac{\lambda}{x}\right)\left(\begin{matrix}\phi_{\sigma}^{(p)}\\ 0\end{matrix}\right)-
B_{p+1}\left(\sum\limits_{\sigma{>}0}\frac{c_{2,\sigma}}{\lambda\sqrt{x}}(-i)^{\ell_{p-1}-1}H_{\ell_{p-1}}^{(1)}\left(\frac{\lambda}{x}\right)\left(\begin{matrix}0 \\ \xi_{\sigma}^{(p)}\end{matrix}\right)\right) 
\end{align*}
and
\begin{align}
&\left(\frac{\pi\lambda}{2}\right)^{-\frac{1}{2}}\Pi_p\tilde{h}_{\lambda}^{(2)}(\Phi)(x,\theta)=\\& \nonumber \sum\limits_{\sigma\geq 0}\frac{c_{1,\sigma}}{\sqrt{x}}(-i)^{\ell_p}H_{\ell_p}^{(2)}\left(\frac{\lambda}{x}\right)\left(\begin{matrix}\phi_{\sigma}^{(p)}\\ 0\end{matrix}\right)
-B_{p+1}\left(\sum\limits_{\sigma{>}0}\frac{c_{2,\sigma}}{\lambda\sqrt{x}}(-i)^{\ell_{p-1}-1}H_{\ell_{p-1}}^{(2)}\left(\frac{\lambda}{x}\right)\left(\begin{matrix}0 \\ \xi_{\sigma}^{(p)}\end{matrix}\right)\right) 
\end{align}
whenever these converge in $C^{\infty}_{\mathrm{loc}}(\widetilde M_0;\Lambda^{\bullet}T^* \widetilde M_0)$. The subcomponents are labelled as 
\begin{align*}
&\left(\frac{\pi\lambda}{2}\right)^{-\frac{1}{2}}\Pi_p\tilde{h}_{1,\lambda}^{(1)}(\Phi)(x,\theta)=
\sum\limits_{\sigma\geq 0}\frac{c_{1,\sigma}}{\sqrt{x}}(-i)^{\ell_{p}}H_{\ell_{p}}^{(1)}\left(\frac{\lambda}{x}\right)\left(\begin{matrix}\phi_{\sigma}^{(p)}\\ 0\end{matrix}\right)
\end{align*}
and
\begin{align*}
&\left(\frac{\pi\lambda}{2}\right)^{-\frac{1}{2}}\Pi_p\tilde{h}_{2,\lambda}^{(1)}(\Phi)(x,\theta)=
-B_{p+1}\left(\sum\limits_{\sigma{>}0}\frac{c_{2,\sigma}}{\lambda\sqrt{x}}(-i)^{\ell_{p-1}-1}H_{\ell_{p-1}}^{(1)}\left(\frac{\lambda}{x}\right)\left(\begin{matrix}0 \\ \xi_{\sigma}^{(p)}\end{matrix}\right)\right) 
\end{align*} 
and similarly for the complex conjugate.
\end{def1}
Writing $\Delta^p = B_{p+1}A_p + A_{p-1}B_p$ for the operator in \eqref{eqn:Lapl}, one checks that for $f,g\in C^{{\infty}}(0,\epsilon_1)$ depending only on $x$ and $\lambda$,
\begin{align}\label{itszero}
&\Delta^p \left[B_{p+1} \left(\begin{matrix}0 \\ f \xi_{\sigma}^{(p)}\end{matrix}\right) +  \left(\begin{matrix} g \phi_{\sigma}^{(p)}\\ 0\end{matrix}\right) \right]= B_{p+1} \Delta^{p+1}  \left(\begin{matrix}0 \\ f \xi_{\sigma}^{(p)}\end{matrix}\right) + \left(\begin{matrix} (\tilde \Delta_1 g) \phi_{\sigma}^{(p)} \\ 0\end{matrix}\right) \\ 
&= B_{p+1}  \left(\begin{matrix}0 \\ (\tilde\Delta_2^{p+1} f) \xi_{\sigma}^{(p)}\end{matrix}\right) +\left(\begin{matrix}   (\tilde \Delta_1 g) \phi_{\sigma}^{(p)} \\ 0\end{matrix}\right) \nonumber
\end{align}
which underlies the construction throughout.

\begin{def1} \label{def2}
For $\Phi\in L^2(N;\Lambda^pT^*N)$ we let 
\begin{align}
c_{1,\sigma} =\left( \Phi, \phi_{\sigma}^{(p)} \right)_{L^2(N;\Lambda^pT^*N)} \quad c_{2,\sigma} =\left( \Phi,  \xi_{\sigma}^{(p)} \right)_{L^2(N;\Lambda^{p}T^*N)}.
\end{align}
The individual solutions to the Helmholtz equation on the manifold $\widetilde{M}_0$ are  
\begin{align}
&\left(2\pi\lambda\right)^{-\frac{1}{2}}\Pi_p\tilde{j}_{1,\lambda}(\Phi)(x,\theta)=
\sum\limits_{\sigma\geq 0}(-i)^{\ell_{p}}\frac{c_{1,\sigma}}{\sqrt{x}}J_{\ell_p}\left(\frac{\lambda}{x}\right)\left(\begin{matrix}\phi_{\sigma}^{(p)}\\ 0\end{matrix}\right)
\end{align}
and
\begin{align}
&\left(2\pi\lambda\right)^{-\frac{1}{2}}\Pi_p\tilde{j}_{2,\lambda}(\Phi)(x,\theta)=-
B_{p+1}\sum\limits_{\sigma{>}0}\left((-i)^{\ell_{p-1}-1}\frac{c_{2,\sigma}}{\lambda\sqrt{x}}J_{\ell_{p-1}}\left(\frac{\lambda}{x}\right)\left(\begin{matrix}0 \\ \xi_{\sigma}^{(p)}\end{matrix}\right)\right) .
\end{align}

The generalized Bessel function on $\widetilde M_0$ is then
\begin{align}
&\left(2\pi\lambda\right)^{-\frac{1}{2}}\Pi_p\tilde{j}_{\lambda}(\Phi)(x,\theta)=\\& \nonumber
\sum\limits_{\sigma\geq 0}(-i)^{\ell_p}\frac{c_{1,\sigma}}{\sqrt{x}}J_{\ell_p}\left(\frac{\lambda}{x}\right)\left(\begin{matrix}\phi_{\sigma}^{(p)}\\ 0\end{matrix}\right)
-B_{p+1}\sum\limits_{\sigma{>}0}\left((-i)^{\ell_{p-1}-1}\frac{c_{2,\sigma}}{\lambda\sqrt{x}}J_{\ell_{p-1}}\left(\frac{\lambda}{x}\right)\left(\begin{matrix}0 \\ \xi_{\sigma}^{(p)}\end{matrix}\right)\right) .
\end{align}
With this normalization, $\tilde j_{\lambda}=\tilde h^{(1)}_{\lambda}+\tilde h^{(2)}_{\lambda}$, since $H^{(1)}_\nu+H^{(2)}_\nu=2J_\nu$.
\end{def1}

For generic constants $b_1\in\mathbb{R}$ and $b_2\in\mathbb{R}$ we set 
\begin{align}
H^{reg}_{\ell}\left(\frac{\lambda}{x}\right)=\frac{b_1}{\sqrt{x}}H_{\ell}^{(1)}\left(\frac{\lambda}{x}\right)+\frac{b_2}{\sqrt{x}}H_{\ell}^{(2)}\left(\frac{\lambda}{x}\right).
\end{align}
Differentiating via \eqref{derivativebessel} gives
\begin{align}\label{actiondelta}
-x^2\partial_xH^{reg}_{\ell}\left(\frac{\lambda}{x}\right)=\mp x(\ell\mp \frac{1}{2})H_{\ell}^{reg}\left(\frac{\lambda}{x}\right)\pm \lambda H_{\ell\mp 1}^{reg}\left(\frac{\lambda}{x}\right).
\end{align}
The action of $B_{p+1}$ on the Hankel functions is therefore explicitly computable. The Poisson operator $P_{\lambda,\widetilde{M}_0}$ is defined analogously to the scalar case in \cite{christiansen}, Prop.~1.1. To find a generalized eigenfunction $E_{\lambda,\tilde{M}_0}(\Phi)$ satisfying $\delta E_{\lambda,\tilde{M}_0}(\Phi)=0$, one uses forms of type $\omega_{1,p}$ and $\omega_{3,p}=\delta(f_2\,dx\wedge \xi)$ as in Proposition~\ref{propsol} and \eqref{itszero}, normalized to support the functional calculus of the next section. 
\begin{prop}\label{scatteringbase}
We define the Poisson operator $P_{k,\lambda,\widetilde{M}_0}$ for  $k=1,2$ on $\Phi\in{L^2(N;\Lambda^{p}T^*N)}$ {in the corresponding Hodge component} by
\begin{align}
P_{k,\lambda,\widetilde{M}_0}\Phi=\tilde{j}_{k,\lambda}(\Phi)(x,\theta) \quad k=1,2
\end{align}
where $\delta \tilde{j}_{k,\lambda}(\Phi)=0$. 
Then the scattering matrix for $\lambda>0$ has two tangential components {on nonharmonic angular data}:
\begin{align}
\mathcal{S}^t_{1,\widetilde{M}_0}(\lambda)= -ie^{-i\pi\sqrt{\left(\frac{n}{2}-p-1\right)^2+\Delta_N}} \qquad \quad \mathcal{S}^t_{2,\widetilde{M}_0}(\lambda)=ie^{-i\pi\sqrt{\left(\frac{n}{2}-p\right)^2+\Delta_N}} 
\end{align}
acting on the forms on $N$ which are co-closed and closed respectively.
{For harmonic angular data, separately,
\begin{equation}\label{harmonic-phase}
\mathcal S^t_{1,\widetilde M_0}(\lambda)\eta
=\mathcal S^t_{2,\widetilde M_0}(\lambda)\eta
=-i e^{-i\pi|\frac n2-p-1|}\eta,
\qquad \eta\in\mathcal H^p(N),\quad\lambda>0.
\end{equation}}
\end{prop}
\begin{proof}
{The second-component calculation below is for nonharmonic data; its harmonic case follows from $\tilde j_{2,\lambda}(\eta)=-\tilde j_{1,\lambda}(\eta)$.} To verify that this is the correct Poisson operator, we check that $\delta\tilde{j}_{k,\lambda}(\Phi)(x,\theta)=0$ and that
\begin{align}\label{diracpoisson}
\langle \tilde{j}_{k,\lambda}(\Phi),\tilde{j}_{k,\mu}(\Psi)\rangle_{L^2(\widetilde M_0)}=4\pi \lambda\delta(\lambda^2-\mu^2)\langle\Phi, \Psi\rangle_{L^2(N)},
\end{align}
where $\delta$ denotes the Dirac delta in the parameter $\lambda$.
Separation of variables shows that $\tilde{j}_{k,\lambda}(\Phi)(x,\theta)$ solves the Helmholtz equation with eigenvalue $\lambda^2$. Verification of \eqref{diracpoisson} simultaneously yields the two tangential scattering matrix components. We begin by establishing $\delta\tilde{j}_{k,\lambda}(\Phi)(x,\theta)=0$ for $k=1,2$. We have automatically
\begin{align}
\delta\Pi_p^{-1}\left(\sum\limits_{\sigma\geq 0}(-i)^{\ell_p}\frac{c_{1,\sigma}}{\sqrt{x}}J_{\ell_p}\left(\frac{\lambda}{x}\right)\left(\begin{matrix}\phi_{\sigma}^{(p)}\\ 0\end{matrix}\right)\right)=0. 
\end{align}
Terms involving basis vectors of the second type satisfy $\delta\psi=0$ only if they lie in the image of $\delta$: a form $\tilde{\psi}_2=\delta\psi_2$ automatically satisfies $\delta^2\tilde{\psi}_2=0$. It suffices to find second-type vectors of this form that solve the Helmholtz equation. Since we have that 
 \[
\delta(\delta \psi_2) =  ({\Pi}_{p-1}^{-1}  \, B_{p} \, {\Pi}_{p}){\Pi}_{p}^{-1}  \, B_{p+1} \, {\Pi}_{p+1}\psi_2= ({\Pi}_{p-1}^{-1}  \, B_{p}B_{p+1}\, {\Pi}_{p+1})\psi_2=0.
\]
If we express $\tilde{\psi_2}$ in basis vectors then 
\begin{align}
\Pi_{p}\tilde{\psi_2}=B_{p+1}\left(\sum\limits_{\sigma\geq 0}\frac{c_{2,\sigma}}{\lambda\sqrt{x}}(-i)^{\ell_{p-1}-1}J_{\ell_{p-1}}\left(\frac{\lambda}{x}\right)\left(\begin{matrix}0 \\ \xi_{\sigma}^{(p)}\end{matrix}\right)\right)
\end{align}
for the coefficients $c_{2,\sigma}$ defined above. This gives the result that $\delta\tilde{j}_{2,\lambda}(\Phi)=0$. 

Now we examine the generalized eigenfunctions piece by piece. We use the Hankel function asymptotics along with the identity
\begin{align}\label{sum}
2J_{\ell_p}\left(\frac{\lambda}{x}\right)=H^{(1)}_{\ell_p}\left(\frac{\lambda}{x}\right)+H^{(2)}_{\ell_p}\left(\frac{\lambda}{x}\right).
\end{align}
Using the asymptotic expansion from \eqref{eqn:Hankelerror}
\begin{align}\label{eq12}
{H^{(1)}_{\ell_{p}}}\left(\frac{\lambda}{x}\right)=\left(\frac{2x}{\pi\lambda}\right)^{\frac{1}{2}}e^{%
i\omega}(1+O(x)) \quad x\rightarrow 0
\end{align}
where $\omega=\frac{\lambda}{x}-\tfrac{1}{2}\ell_p\pi-\tfrac{1}{4}\pi$, we also have that
\begin{align}\label{ex}
& \left(\frac{\pi\lambda}{2}\right)^{\frac{1}{2}}e^{\frac{i\pi}{4}} \frac{1}{\sqrt{x}} i^{\ell_p}H_{\ell_p}^{(1)} \left(\frac{\lambda}{x}\right)\left(\begin{matrix}\phi_{\sigma}^{(p)}\\ 0\end{matrix}\right)\\ &=e^{ \frac{i\lambda}{x}}\left(\begin{matrix} \phi_{\sigma}^{(p)}\\ 0\end{matrix}\right)(1+O(x)),  \quad \mathrm{as}\quad x\rightarrow 0. \notag
\end{align}
Similarly for the complex conjugate of $H^{(1)}$, $H^{(2)}$, we have that
\begin{align}
& \left(\frac{\pi\lambda}{2}\right)^{\frac{1}{2}}e^{-\frac{i\pi}{4}} \frac{1}{\sqrt{x}} (-i)^{\ell_p}H_{\ell_p}^{(2)} \left(\frac{\lambda}{x}\right)\left(\begin{matrix}\phi_{\sigma}^{(p)}\\ 0\end{matrix}\right)\\ &=e^{-\frac{i\lambda}{x}}\left(\begin{matrix} \phi_{\sigma}^{(p)}\\ 0\end{matrix}\right)(1+O(x)),  \quad \mathrm{as}\quad x\rightarrow 0. \notag
\end{align}
This allows us to compute the action of $\mathcal{S}^t_{\widetilde{M}_{0}}(\lambda)$ in the first component. Using \eqref{ex} the equation above and the normalization constants in the definitions of $\tilde{h}_{\lambda}^{(1)}(\Phi), \tilde{h}_{\lambda}^{(2)}(\Phi)$, we define the outgoing and incoming coefficients as follows:
\begin{align}
C_{+,1}=(-i)^{\ell_p}e^{-\tfrac{i}{2}\ell_p\pi-\tfrac{i}{4}\pi}\qquad C_{-,1}=(-i)^{\ell_p}e^{\tfrac{i}{2}\ell_p\pi+\tfrac{i}{4}\pi}.
\end{align}
As a result we have that
\begin{align}
\frac{\mathrm{outgoing}}{\mathrm{incoming}}=\frac{C_{+,1}}{C_{-,1}}=\frac{(-i)^{\ell_p}e^{-\tfrac{i}{2}\ell_p\pi-\tfrac{i}{4}\pi}}{(-i)^{\ell_p}e^{\tfrac{i}{2}\ell_p\pi+\tfrac{i}{4}\pi}}=e^{-i\ell_p\pi-\tfrac{i}{2}\pi}=\mathcal{S}^t_{1,\tilde{M}_0}(\lambda).
\end{align}

Moreover, using the definition of $B_{p+1}$ and~\eqref{derivativebessel},
\begin{equation}\label{eq11}
\begin{split}
& B_{p+1}\left( \frac{1}{ \lambda \sqrt{x}} H_{\ell_{ p-1}}^{(2)}\left(\frac{\lambda}{x}\right)\left(\begin{matrix}0 \\ \xi_{\sigma}^{(p)}\end{matrix}\right)\right)\\
&= \left[ -x^2 \partial_x +\left( \tfrac{n-1}{2}-p\right) x \right] \left( \frac{1}{\lambda \sqrt{x}} H_{\ell_{ p-1}}^{(2)} \right)\left(\begin{matrix} \xi_{\sigma}^{(p)} \\0 \end{matrix}\right) -\frac{ \sqrt{\sigma} \,\sqrt{x}}{ \lambda } H_{\ell_{ p-1}}^{(2)}\left(\begin{matrix}0\\ \phi_{\sigma}^{(p-1)} \end{matrix}\right) \\&
=\left( \left[ \tfrac n2 -\ell_{ p-1} -p \right] \frac{\sqrt{x}}{\lambda} H_{\ell_{ p-1}}^{(2)} + \frac{1}{\sqrt{x}} H_{\ell_{ p-1}-1}^{(2)} \right) \left(\begin{matrix} \xi_{\sigma}^{(p)} \\0 \end{matrix}\right) -\frac{ \sqrt{\sigma} \,\sqrt{x}}{ \lambda } H_{\ell_{ p-1}}^{(2)}\left(\begin{matrix}0\\ \phi_{\sigma}^{(p-1)} \end{matrix}\right) \\
&=\left( \left[ \tfrac n2 -p \right] \frac{\sqrt{x}}{\lambda} H_{\ell_{ p-1}}^{(2)} + \frac{1}{\sqrt{x}} \left(H_{\ell_{ p-1}}^{(2)} \right)'\right) \left(\begin{matrix} \xi_{\sigma}^{(p)} \\0 \end{matrix}\right) -\frac{ \sqrt{\sigma} \,\sqrt{x}}{ \lambda } H_{\ell_{ p-1}}^{(2)}\left(\begin{matrix}0\\ \phi_{\sigma}^{(p-1)} \end{matrix}\right).
\end{split}
\end{equation} 
As a result, 
\begin{align}\label{ex2}
&\left(\frac{\pi\lambda}{2}\right)^{\frac{1}{2}}e^{-\frac{i\pi}{4}}B_{p+1}\left( \frac{1}{\lambda\sqrt{x}}(-i)^{\ell_{p-1}(\sigma)-1} H_{\ell_{p-1}(\sigma)
}^{(2)}\left(\frac{\lambda}{x}\right)\left(\begin{matrix}0 \\ \xi_{\sigma}^{(p)}\end{matrix}\right)\right)\\&=e^{-\frac{i\lambda}{x}}\left(\begin{matrix} \xi_{\sigma}^{(p)}\\ 0\end{matrix}\right)(1+O(x)), \quad \mathrm{as} \quad x\rightarrow 0\notag
\end{align} 
and similarly for the complex conjugate $H^{(1)}$. Then the remaining criterion is checked by using \eqref{ex2} as well as the definition of the coefficients $c_{1,\sigma}$ and $c_{2,\sigma}$.  To check the equality \eqref{diracpoisson} we observe that using \eqref{eqn:Hankelerror}
this last inner product can be computed by taking the limit where $\chi_x(x)$ is a cutoff function with support away from $x=0$:
$$
 \lim_{x \to 0} \frac{1}{\lambda^{2} - \mu^2}\left( \langle  \Delta_p \tilde{j}_{k,\lambda}(\Phi), \chi_x  \tilde{j}_{k,\mu}(\Psi)\rangle_{L^2(\widetilde M_0)} - \langle \tilde{j}_{k,\lambda}(\Phi), \chi_x \Delta_p \tilde{j}_{k,\mu}(\Psi)\rangle_{L^2(\widetilde M_0)} \right), 
$$
and using Green's identity to arrive at 
\begin{equation} \label{eq13}
\begin{split}
 \lim_{x \to 0}  \frac{1}{\lambda^{2}-\mu^{2}}& \left(\; \int\limits_{\partial \widetilde M_0}\langle x^{-2}dx \wedge \tilde{j}_{k,\lambda}(\Phi),d\tilde{j}_{k,\mu}(\Psi)\rangle_{\widetilde M_0}\,x^{1-n}\,dS \right.\\
 &\qquad \left. -\int\limits_{\partial \widetilde M_0}\langle d\tilde{j}_{k,\lambda}(\Phi),x^{ -2 }dx \wedge \tilde{j}_{k,\mu}(\Psi)\rangle_{\widetilde M_0}\,x^{1-n}\,dS\right).
\end{split}
\end{equation}
In the above inner products, $dS$ is the volume element of $N$, unweighted. Moreover only the derivative with respect to $\partial_x$ will survive after taking the inner product, and since $|dx|_{\widetilde{M}_{0}}=x^2$ the expression simplifies to
\begin{equation}\begin{split} \label{eq10}
 \lim_{x \to 0}  \frac{1}{\lambda^{2}-\mu^{2}}& \left( \;\int\limits_{\partial \widetilde M_0}\langle  \tilde{j}_{k,\lambda}(\Phi), x^{2} \partial_x \left(\tilde{j}_{k,\mu}(\Psi)\right)\rangle_{N}\,x^{1-n+2p} \,  dS \right.
 \\
 &\qquad \left. -\int\limits_{\partial \widetilde M_0}\langle x^2 \partial_x \left(\tilde{j}_{k,\lambda}(\Phi)\right),  \tilde{j}_{k,\mu}(\Psi)\rangle_{N}\,x^{1-n+2p} \,dS\right).
\end{split}
\end{equation}
For the case $k=1$, the first integral in \eqref{eq10} will contribute with the term 
\begin{align*} 
 (4\pi^2\lambda\mu)^{1/2} &\sum\limits_{\sigma\geq 0} \sum\limits_{\tilde{\sigma}\geq 0}  \left[ (- i)^{\ell_{p}(\sigma)} i^{\ell_{p}(\tilde{\sigma})} \frac{c_{1,\sigma}(\Phi)}{\sqrt{x}}J_{\ell_p(\sigma)}\left(\frac{\lambda}{x}\right) \; x^2 \partial_x \left(\overline{\frac{c_{1,\tilde{\sigma}}(\Psi)}{\sqrt{x}}J_{\ell_p(\tilde{\sigma})}\left(\frac{\mu}{x}\right)} \right) \right.\\
&   \cdot \int\limits_{\partial \widetilde M_0} \left\langle  \phi_{\sigma}^{(p)} ,    \phi_{\tilde{\sigma}}^{(p)}  \right\rangle_N \Bigr]\; dS \\
= (4\pi^2\lambda\mu)^{1/2}  &   \sum\limits_{\sigma\geq 0}   \left(c_{1,\sigma}(\Phi) \, \overline{c_{1, {\sigma}}(\Psi)}  \;  \frac{1}{\sqrt x}J_{\ell_p(\sigma)}\left(\frac{\lambda}{x}\right)    \,  x^2 \partial_x \left( \frac{1}{\sqrt x} \overline{J_{\ell_p( {\sigma})}\left(\frac{\mu}{x}\right)} \right) \;   \right)
\end{align*}
(there is an additional term with a constant in front that will disappear after we take the difference with the second integral and which we omit) where we have used the fact that the $\phi_\sigma$ are an orthonormal basis of $L^2(N)=L^2(\partial \widetilde M_0)$ to reduce the double sum to a single sum over $\sigma$. The second integral will also contribute with a similar term.
 
Observe that
\begin{align} \label{thirdbi}
&\frac{1}{\sqrt x}H_{\ell_p(\sigma)}^{(1)}\left(\frac{\lambda}{x}\right)    \,  x^2 \partial_x \left( \frac{1}{\sqrt x} \overline{H_{\ell_p( {\sigma})}^{(1)}\left(\frac{\mu}{x}\right)} \right)  -  x^2 \partial_x \left(  \frac{1}{\sqrt x}H_{\ell_p(\sigma)}^{(1)}\left(\frac{\lambda}{x}\right)   \right)  \,  \frac{1}{\sqrt x} \overline{H_{\ell_p( {\sigma})}^{(1)}\left(\frac{\mu}{x}\right)} \\ \nonumber
=&- \frac{\mu}{x} \,  H_{\ell_p(\sigma)}^{(1)}\left(\frac{\lambda}{x}\right)    \,   \overline{\left(  H_{\ell_p( {\sigma})}^{(1)}\right)' \left(\frac{\mu}{x}\right)} +   \frac{\lambda}{x}  \left( H_{\ell_p(\sigma)}^{(1)}\right)'\left(\frac{\lambda}{x}\right)    \,   \overline{H_{\ell_p( {\sigma})}^{(1)}\left(\frac{\mu}{x}\right)}\\ \nonumber
=& \frac{2i(\mu+\lambda)}{\pi \sqrt{\lambda\mu}}\, e^{i\frac{(\lambda-\mu)}{x}}(1+O(x))
\end{align}
where we have used that 
\begin{equation}
\begin{split}\label{difasymp}
&\left(H^{(1)}_{\ell_p}\right)'\left(\frac{\lambda}{x}\right)=i\, \left( \frac{2x}{\pi \lambda}\right)^{1/2} e^{i\omega}(1+O(x))\\&
\left(H^{(2)}_{\ell_p}\right)'\left(\frac{\lambda}{x}\right)=-i\, \left( \frac{2x}{\pi \lambda}\right)^{1/2} e^{-i\omega}(1+O(x))
\end{split}
\end{equation}
with $\omega=\frac{\lambda}{x}-\tfrac{1}{2}\ell_p\pi-\tfrac{1}{4}\pi$, and the corresponding asymptotics for $H^{(1)}, H^{(2)}$ from~\eqref{eq12}.

We turn to $k=2$. Using formula \eqref{eq11} write
\[
B_{p+1}\left( \frac{1}{ \lambda \sqrt{x}} H_{\ell_{ p-1}}^{(1)}\left(\frac{\lambda}{x}\right)\left(\begin{matrix}0 \\ \xi_{\sigma}^{(p)}\end{matrix}\right)\right) = f_{\lambda,\sigma} \left(\begin{matrix}  \xi_{\sigma}^{(p)} \\ 0 \end{matrix}\right) + g_{\lambda,\sigma} \left(\begin{matrix}0\\ \phi_{\sigma}^{(p-1)} \end{matrix}\right)
\]
and observe that
\[
x^{-2} dx \wedge \Pi_{p}^{-1} B_{p+1}\left( \frac{1}{ \lambda \sqrt{x}} H_{\ell_{ p-1}}^{(1)}\left(\frac{\lambda}{x}\right)\left(\begin{matrix}0 \\ \xi_{\sigma}^{(p)}\end{matrix}\right)\right)=x^{-2} dx \wedge \Pi_p^{-1} f_{\lambda,\sigma} \left(\begin{matrix} \xi_{\sigma}^{(p)} \\0 \end{matrix}\right)
\]
and
\begin{equation*} 
\begin{split}
&  \Pi_{p+1} d \left( \Pi_{p}^{-1} B_{p+1}\left( \frac{1}{ \lambda \sqrt{x}} H_{\ell_{ p-1}}^{(1)}\left(\frac{\lambda}{x}\right)\left(\begin{matrix}0 \\ \xi_{\sigma}^{(p)}\end{matrix}\right)\right) \right) =A_p \left[   f_{\lambda,\sigma} \left(\begin{matrix}  \xi_{\sigma}^{(p)} \\ 0 \end{matrix}\right) + g_{\lambda,\sigma} \left(\begin{matrix}0\\ \phi_{\sigma}^{(p-1)} \end{matrix}\right) \right]\\
&= \left(\begin{matrix}0\\ (x^2\partial_x +\left(\frac{n-1}{2}-p\right)\, x) f_{\lambda,\sigma} \, \xi_{\sigma}^{(p)}  -x \, g_{\lambda,\sigma} \, d_N \phi_{\sigma}^{(p-1)} \end{matrix} \right)
\end{split} 
\end{equation*}
since the $\xi_{\sigma}^{(p)} $ are closed. 

Hence, for $k=2$ the first boundary integral in \eqref{eq13} becomes  
\begin{equation}\label{eq14}
\begin{split}
&\int\limits_{\partial \widetilde M_0}\langle  x^{-2}dx \wedge \tilde{h}^{(1)}_{2,\lambda}(\Phi),d\tilde{h}^{(1)}_{2,\mu}(\Psi)\rangle_{\widetilde M_0}\, x^{1-n}\,dS\\
&=\left(\tfrac{\pi\lambda}{2}\right)^{1/2}\!\left(\tfrac{\pi\mu}{2}\right)^{1/2}\sum\limits_{\sigma\geq 0} \sum\limits_{\tilde{\sigma}\geq 0} \left( i^{\ell_{p-1}(\sigma)-1} (-i)^{\ell_{p-1}(\tilde{\sigma})-1}  c_{2,\sigma}(\Phi) \,\overline{c_{2,\tilde{\sigma}}(\Psi)}   f_{\lambda,\sigma} \;  \right.\\
& \cdot \int\limits_{\partial \widetilde M_0} \left\langle \xi_{\sigma}^{(p)}   ,    (x^2\partial_x +\left(\tfrac{n-1}{2}-p\right)\, x) f_{\mu,\tilde\sigma} \, \xi_{\tilde\sigma}^{(p)}  -x \, g_{\mu,\tilde\sigma} \, d_N \phi_{\tilde\sigma}^{(p-1)} \right\rangle_N \Bigg)  \; dS \\
&= \left(\tfrac{\pi\lambda}{2}\right)^{1/2}\!\left(\tfrac{\pi\mu}{2}\right)^{1/2}   \sum\limits_{\sigma\geq 0}    c_{2,\sigma}(\Phi) \, \overline{c_{2, {\sigma}}(\Psi)}  \;    f_{\lambda,\sigma} \, \left[(x^2\partial_x +\left(\tfrac{n-1}{2}-p\right)\, x)\overline{ f_{\mu, \sigma}}   -x \, \overline{g_{\mu, \sigma}}  \right] 
\end{split}
\end{equation}
since $\left\langle \xi_{\sigma}^{(p)}   ,      d_N \phi_{\tilde\sigma}^{(p-1)} \right\rangle_{L^2(N)} \neq 0$ if and only if  $d_N \phi_{\tilde\sigma}^{(p-1)}$, which is a closed $p$-form with eigenvalue $\tilde{\sigma}$, coincides with  $\xi_{\sigma}^{(p)}.$

Taking into account the large-argument asymptotic expansion of the Hankel functions as $x\to 0$, we find
\begin{align}\label{leadingrot}
(x^2\partial_x +\left(\tfrac{n-1}{2}-p\right)\, x)\overline{ f_{\mu, \sigma}}   -x \, \overline{g_{\mu, \sigma}} \sim - \frac{\mu}{\sqrt{x}} 
\overline{ (H_{\ell_{ p-1}-1}^{(1)})' \left(\frac{\mu}{x}\right)} 
\end{align}
hence 
\begin{align}\label{secondbi}
\begin{split}
&\int\limits_{\partial \widetilde M_0}\langle  x^{-2}dx \wedge \tilde{h}_{2,\lambda}^{(1)}(\Phi),d\tilde{h}^{(1)}_{2,\mu}(\Psi)\rangle_{\widetilde M_0}\, x^{1-n}\,dS\\
&= \left(\tfrac{\pi\lambda}{2}\right)^{1/2}\!\left(\tfrac{\pi\mu}{2}\right)^{1/2}   \sum\limits_{\sigma\geq 0}    c_{2,\sigma}(\Phi) \,\overline{ c_{2, {\sigma}}(\Psi) } \;    f_{\lambda,\sigma} \, \left[(x^2\partial_x +\left(\tfrac{n-1}{2}-p\right)\, x)\overline{ f_{\mu, \sigma}}   -x \, \overline{g_{\mu, \sigma}}  \right] \\
&= i \mu   \;   e^{i\frac{(\lambda-\mu)}{x}} \, (\langle\Phi, \Psi\rangle_{L^2(N;\Lambda^{p}T^*N)}+O(x)) .
\end{split}
\end{align}
Similarly the second integral becomes
\begin{align}\label{firstbi}
\begin{split}
&\int\limits_{\partial \widetilde M_0}\langle d\tilde{h}^{(1)}_{ 2,\lambda}(\Phi),  x^{-2}dx \wedge \tilde{h}^{(1)}_{2,\mu}(\Psi)\rangle_{ \widetilde M_0}\, x^{1-n}\, dS \\
&=\left(\tfrac{\pi\lambda}{2}\right)^{1/2}\!\left(\tfrac{\pi\mu}{2}\right)^{1/2}   \sum\limits_{\sigma\geq 0}    c_{2,\sigma}(\Phi) \, \overline{c_{2, {\sigma}}(\Psi)}\;    \left[(x^2\partial_x +\left(\tfrac{n-1}{2}-p\right)\, x) f_{\lambda, \sigma}   -x \, g_{\lambda, \sigma}  \right]  \, \overline{f_{\mu,\sigma}} \\
&=- i \lambda \;  e^{i\frac{(\lambda-\mu)}{x}}\, (\langle\Phi, \Psi\rangle_{L^2(N;\Lambda^{p}T^*N)}+O(x)) . 
\end{split}
\end{align}
Combining \eqref{eq13} with the decomposition \eqref{sum}, the oscillatory cross-terms cancel by the same symmetry argument as for $\tilde{j}_{1,\lambda}$, and the remaining diagonal term gives the desired result. We define the outgoing and incoming coefficients for $k=2$ as:
\begin{align}
C_{+,2}=(-i)^{\ell_{p-1}-1}e^{-\tfrac{i}{2}(\ell_{p-1}-1)\pi-\tfrac{i}{4}\pi}\qquad C_{-,2}=(-i)^{\ell_{p-1}-1}e^{\tfrac{i}{2}(\ell_{p-1}-1)\pi+\tfrac{i}{4}\pi}.
\end{align}
We compute the scattering matrix coefficient from the leading-order terms in \eqref{leadingrot} and \eqref{difasymp},
\begin{align}
\frac{\mathrm{outgoing}}{\mathrm{incoming}}=\frac{C_{+,2}}{C_{-,2}}=\frac{(-i)^{\ell_{p-1}-1}e^{-\tfrac{i}{2}(\ell_{p-1}-1)\pi-\tfrac{i}{4}\pi}}{(-i)^{\ell_{p-1}-1}e^{\tfrac{i}{2}(\ell_{p-1}-1)\pi+\tfrac{i}{4}\pi}}=e^{-i(\ell_{p-1}-1)\pi-\tfrac{i}{2}\pi}=\mathcal{S}^t_{2,\tilde{M}_0}(\lambda).
\end{align}

\end{proof}
\begin{rem}
In the case when $n=3, p=1$ with the Euclidean metric on the exact manifold $\widetilde{M}_0$ the generalized eigenfunctions coincide with those in \cite{YFW} for Maxwell's equations.  For further background on the quantization of electromagnetic fields on non-compact manifolds we refer to \cite{stratton,AS}.
\end{rem}

We now extend the construction to $M$ and $\widetilde{M}$ Recall that the metric on $\widetilde{M}$ is a compact perturbation of the metric $g_0$ on $\widetilde{M}_0$, and the metric on $M$ satisfies either \eqref{closeness} or \eqref{closeness2}. Let $\chi\in C^{\infty}(M)$ be a cutoff function supported in $Z=(0,\epsilon_1)\times N$  with $\chi=1$ on $\{0<x<\epsilon_1/2\}$; in particular $\chi=1$ near $x=0$, so that $\chi\tilde{j}_{\lambda}(\Phi)$ defines the model eigenfunctions near the boundary at infinity. We recall the following definition.
\begin{def1}
A generalized eigenfunction $E_{\lambda,M}(\Phi)$ on $M$ with $\delta E_{\lambda,M}(\Phi)=0$ is given by: 
\begin{align}
E_{\lambda,M}(\Phi)=\chi\tilde{j}_{\lambda}(\Phi)-R_{\lambda}(\Delta_M-\lambda^2)(\chi \tilde{j}_{\lambda}(\Phi))
\end{align}
for $R_\lambda$ as defined in the introduction.
\end{def1} 
We remark that, on the model manifold $\widetilde{M}_0$ we have $E_{\lambda,\widetilde{M}_0}(\Phi)=\tilde{j}_{\lambda}(\Phi)$. Co-closedness of $E_{\lambda,M}(\Phi)$ follows immediately. Note that $\delta(\chi\tilde{j}_{\lambda}(\Phi))=[\delta,\chi]\tilde{j}_{\lambda}(\Phi)$ is compactly supported, since $\tilde{j}_{\lambda}(\Phi)$ is co-closed and $d\chi$ has compact support. Since $R_{\lambda}(\Delta_M-\lambda^2)=\Id$ on $L^2_{comp}$ we have $\delta(\chi\tilde{j}_{\lambda}(\Phi))-R_{\lambda}^{p-1}(\Delta_{p-1}-\lambda^2)(\delta(\chi\tilde{j}_{\lambda}(\Phi)))=0$.  This gives $E_{\lambda,M}(\Phi)$ is co-closed. The same argument shows that $E_{\lambda,M}(\Phi)$ is an eigenfunction on $M$.  Moreover, the same construction holds with $\widetilde{M}$ replacing $M$. 
We also have the following equality in $Z$.
\begin{prop}\label{genrepA}
For $\lambda>0$, there exists a unique $A^{\mathrm{abs}}_{\lambda}(\Phi)\in C^{\infty}(N,\Lambda^pT^*N)$ such that the generalized eigenfunction on $Z$ with $\delta E_{\lambda,M}(\Phi)(x,\theta)=0$ is given by 
\begin{align}  \label{genrepAeq}
E_{\lambda,M}(\Phi)(x,\theta)|_Z=\tilde{j}_{\lambda}(\Phi)(x,\theta)+\tilde{h}^{(1)}_{\lambda}(A^{\mathrm{abs}}_{\lambda}\Phi)(x,\theta)+O(x^{\frac{n+1}{2}})
\end{align}
where $\tilde{h}^{(1)}_{\lambda}(\Phi)(x,\theta)$ is as in Definition~\ref{def:hankel-solutions}.
\end{prop}
\begin{proof}
For simplicity we denote the Laplacian on $M$ by $\Delta$. Following \cite{christiansen}, Corollary~1.1, for $\Phi\in C^{\infty}(N;\Lambda^{\bullet} T^*N)$
\begin{align}
(\Delta-\lambda^2)\chi E_{\lambda,\widetilde{M}_0}(\Phi)=\chi(\Delta-\Delta_{\widetilde{M}_0})E_{\lambda,\widetilde{M}_0}(\Phi)+[\Delta,\chi] E_{\lambda,\widetilde{M}_0}(\Phi)\in C^{\infty}(M;\Lambda^{\bullet} T^*M).
\end{align} 
Using the assumptions on the metric \eqref{closeness}, the result is that $(\Delta-\lambda^2)\chi E_{\lambda,\widetilde{M}_0}(\Phi)$ is smooth and vanishes to infinite order as $x\rightarrow 0$. The assumptions on the metrics are the exact assumptions needed to apply Proposition~14 of \cite{melbook2}. The positive-eigenform obstruction in that proposition is absent by Proposition~\ref{prop:no-positive-eigenforms}; hence there is a $\Theta \in C^{\infty}(M;\Lambda^{\bullet} T^*M)$ such that
\begin{align}
(\Delta-\lambda^2)e^{i\frac{\lambda}{x}}x^{\frac{n-1}{2}}\Theta=(\Delta-\lambda^2)\chi E_{\lambda,\widetilde{M}_0}(\Phi).
\end{align}
The uniqueness required here follows from Proposition~\ref{prop:no-positive-eigenforms}.
The $p$-form $\Theta$ has leading order term coinciding with an outgoing wave. Using the Appendix, we have that the Hankel functions $H^{(1)}_{\ell}(\frac{\lambda}{x})$ can be replaced by their asymptotic expansion \eqref{eqn:Hankelerror} as $x \to 0$. This allows us to regroup and relabel the $\Theta$. Indeed expanding the Hankel function in the generalized eigenfunction definition
\begin{align}
e^{i\frac{\lambda}{x}}x^{\frac{n-1}{2}}\Theta=e^{i\frac{\lambda}{x}}(x^{\frac{n-1}{2}}+O(x^{\frac{n+1}{2}}))\Phi'
\end{align}
where $\Phi'=C_+A^{\mathrm{abs}}_{\lambda}\Phi$ and $C_+$ is the generalized Hankel function normalization constant.  The desired result follows by defining $\Theta(x)|_{x=0}=C_+A^{\mathrm{abs}}_{\lambda}\Phi$, which by the same Prop.~14 is unique. We conclude that on the support of $\chi$ we have that 
\begin{align}
\tilde{h}_{\lambda}^{(1)}(A^{\mathrm{abs}}_{\lambda}(\Phi))=-R_{\lambda}([\Delta,\chi]+\chi(\Delta-\Delta_{\widetilde{M}_0})) \tilde{j}_{\lambda}(\Phi)+O(x^{\frac{n+1}{2}}).
\end{align} 
  The term $([\Delta,\chi]+\chi(\Delta-\Delta_{\widetilde{M}_0}))$ is $O(x^{\infty})$ under Assumption \ref{closeness} and $O(x^{n_0-1}\nabla+x^{n_0})$ under Assumption \ref{closeness2}. 
\end{proof}  
Under Assumption \eqref{closeness2}, the proof still holds but $\Theta$ is no longer smoothing, the restriction to $x=0$ is a map $H^{-s}(N;\Lambda^pT^*N)\rightarrow L^2(N;\Lambda^pT^*N)$ where $s<n_0-1$. This is implicit in the proof of Prop~14 and Remark 3, but is spelled out formally in \cite{JSB}, Prop 2.2 and Prop 2.3. 
The same proof also applies for $\widetilde{M}$. 
As defined, the generalized eigenfunctions $E_{\lambda,M}(\Phi)$ are distributions in $\lambda$ valued in $\mathcal{S}(M;\Lambda^p T^*M)$. They share the properties established in \cite{OS} and extend the definitions made there. 
This gives two different but equivalent normalizations.  The Hankel-normalized, or absolute, scattering matrix is
\[
S^{M,\mathrm{abs}}(\lambda)=\Id+A^{M,\mathrm{abs}}_{\lambda}.
\]
The phase-normalized tangential scattering matrix from Definition~\ref{def:phase-scattering} is obtained by conjugating with the incoming and outgoing Hankel phase factors:
\begin{align}\label{phase-abs-relation}
\mathcal{S}^t_{M}(\lambda)=C_+S^{M,\mathrm{abs}}(\lambda)C_{-}^{-1}
=C_+(\Id+A^{M,\mathrm{abs}}_{\lambda})C_{-}^{-1}.
\end{align}
For the exact model end one has
\begin{align}
\mathcal{S}^t_{\widetilde M_0}(\lambda)=C_+C_{-}^{-1}.
\end{align}
Here $C_-$ and $C_+$ are the diagonal unitary maps which convert the Hankel incoming and outgoing coefficients into the phase coefficients used in Definition~\ref{def:phase-scattering}; in particular, all constant Hankel phases are contained in $C_\pm$.  Thus the smoothing remainder in Theorem~\ref{lengths} is not $A^{\mathrm{abs}}_{\lambda}$ itself, but rather
\begin{align}
\alpha_M(\lambda)
:=\mathcal{S}^t_M(\lambda)-\mathcal{S}^t_{\widetilde M_0}(\lambda)
=C_+A^{\mathrm{abs}}_{M,\lambda}C_-^{-1}.
\end{align}
We also have that the difference obeys
\begin{align}\label{alphatilde}
\tilde{\alpha}(\lambda)=\mathcal{S}^t_{M}(\lambda)-\mathcal{S}^t_{\widetilde M}(\lambda)=C_+\left(A^{\mathrm{abs}}_{M,\lambda}-A^{\mathrm{abs}}_{\tilde{M},\lambda}\right)C_{-}^{-1}.
\end{align}

\begin{prop}\label{uniS} 
Let $\lambda>0$ and let $A^{\mathrm{abs}}_\lambda$ be the Hankel-normalized scattering amplitude from Proposition \ref{genrepA}.  Define the absolute scattering matrix on $M$ as
\[
        S^{\mathrm{abs}}(\lambda):=\Id+A^{\mathrm{abs}}_\lambda
        \quad \text{on } C^\infty(N;\Lambda^pT^*N).
\]
Then $S^{\mathrm{abs}}(\lambda)$ extends uniquely to a unitary operator on $L^2(N;\Lambda^pT^*N)$. Equivalently,
\[
        \langle S^{\mathrm{abs}}(\lambda)\Phi,S^{\mathrm{abs}}(\lambda)\Psi\rangle_{L^2(N)}
        =\langle \Phi,\Psi\rangle_{L^2(N)}
        \qquad \text{for all } \Phi,\Psi\in C^\infty(N;\Lambda^pT^*N).
\]
Moreover the analytic continuation satisfies
\[
        S^{\mathrm{abs}}(\lambda)^*=S^{\mathrm{abs}}(\overline{\lambda})^{-1}
        \qquad (\operatorname{Im}\lambda\geq 0),
\]
and hence $S^{\mathrm{abs}}(\lambda)^*=S^{\mathrm{abs}}(\lambda)^{-1}$ for real positive $\lambda$.
\end{prop}
\begin{proof}
We begin by recording the asymptotic normalization.  By Proposition \ref{genrepA}, on the end $Z$ the generalized eigenfunction has an incoming part and an outgoing part.   The expression \eqref{genrepAeq}   can be written as
\begin{align*}
E_{\lambda,M}(\Phi)(x,\theta)|_Z=\tilde{h}^{(2)}_{\lambda}(\Phi)(x,\theta)+\tilde{h}^{(1)}_{\lambda}(S^{\mathrm{abs}}(\lambda)\Phi)(x,\theta)
\end{align*}
which is what we will use for the proof. Indeed, using \eqref{ex} and \eqref{ex2} this expansion becomes 
\begin{align}\label{expansion}
\Pi_pE_{\lambda,M}(\Phi)(x,\theta)
&=e^{-i\lambda/x}C_-\Phi
  +e^{i\lambda/x}C_+S^{\mathrm{abs}}(\lambda)\Phi+O(x)
\end{align}
as $x\downarrow0$. Equivalently, if $F=C_-(\lambda)\Phi$ is the phase-normalized incoming coefficient, then
\begin{align}\label{phase-from-abs}
\mathcal{S}^t_M(\lambda)F
=C_+S^{\mathrm{abs}}(\lambda)C_-^{-1}F.
\end{align}
The principal idea is then that the cross terms cancel in the limit of the boundary pairing. 
Fix $\Phi,\Psi\in C^\infty(N;\Lambda^pT^*N)$ and set
\[
        u=E_{\lambda,M}(\Phi),\qquad v=E_{\lambda,M}(\Psi).
\]
For $\epsilon>0$ let
\[
        M_\epsilon:=\{(x,\theta)\in M: x\geq \epsilon\}.
\]
Since $(\Delta_p-\lambda^2)u=(\Delta_p-\lambda^2)v=0$ and the forms are co-closed, Green's identity for the Hodge Laplacian gives  
\begin{align}\label{greenuni}
0 &= \, \langle \Delta_pu,v\rangle_{M_\epsilon}
  -\langle u,\Delta_pv\rangle_{M_\epsilon} \\  \nonumber
&=
\int_{\partial M_\epsilon}
   \left\langle x^{-2}dx\wedge u,dv\right\rangle x^{1-n}\,dS_{\epsilon}
-
\int_{\partial M_\epsilon}
   \left\langle du,x^{-2}dx\wedge v\right\rangle x^{1-n}\,dS_{\epsilon}
\end{align}
where $dS_\epsilon$ is the volume form on $N$ with respect to the $h(\epsilon)$ metric.
We relabel the boundary pairing 
\begin{align}\label{bpairing}
b_{\epsilon}(u,v)=\int_{\partial M_\epsilon}
   \left\langle x^{-2}dx\wedge u,dv\right\rangle x^{1-n}\,dS_{\epsilon}
-
\int_{\partial M_\epsilon}
   \left\langle du,x^{-2}dx\wedge v\right\rangle x^{1-n}\,dS_{\epsilon}. 
\end{align} 
 Substituting \eqref{expansion} into \eqref{greenuni} gives four leading terms for $b_{\epsilon}(u,v)$. The possible cases we need to analyze are 
 \begin{align*}
 b_{\epsilon}(\tilde{h}_{k,\lambda}^{(1)}(\Phi), \tilde{h}_{m,\overline{\lambda}}^{(1)}(\Psi)), \quad b_{\epsilon}(\tilde{h}_{k,\lambda}^{(1)}(\Phi),\tilde{h}_{m,\overline{\lambda}}^{(2)}(\Psi)) \quad \text{for} \quad k,\,m=1,2
 \end{align*}
Given  the behavior of the eigenfunctions on the boundary as reflected in \eqref{expansion}, the mixed terms for $k\neq m$ vanish. The remaining boundary pairings have already been computed symbolically as $\epsilon \to 0$ as part of the previous proof as \eqref{thirdbi}, \eqref{secondbi}, and \eqref{firstbi}, provided that we select either $\mu=\lambda$ or $\mu=-\lambda$, respectively.  We see this in the exact computations of the bilinear forms in \eqref{thirdbi}, \eqref{secondbi}, and \eqref{firstbi}, and we obtain that their sum corresponds to $\delta_0(\lambda-\mu)$ as $\epsilon \rightarrow 0$, exactly as computed before.
Using these explicit computations, we see that
\begin{align}
0
&=2i\lambda\left(
      \langle C_-\Phi,C_-\Psi\rangle_{L^2(N)}
      -\langle C_+S^{\mathrm{abs}}(\lambda)\Phi,C_+S^{\mathrm{abs}}(\lambda)\Psi\rangle_{L^2(N)}
     \right)\notag\\
&=2i\lambda\left(
      \langle \Phi,\Psi\rangle_{L^2(N)}
      -\langle S^{\mathrm{abs}}(\lambda)\Phi,S^{\mathrm{abs}}(\lambda)\Psi\rangle_{L^2(N)}
     \right).
\end{align}
Since $\lambda>0$, this proves
\[
        \langle S^{\mathrm{abs}}(\lambda)\Phi,S^{\mathrm{abs}}(\lambda)\Psi\rangle_{L^2(N)}
        =\langle \Phi,\Psi\rangle_{L^2(N)}
        \qquad
        \Phi,\Psi\in C^\infty(N;\Lambda^pT^*N).
\]
Thus $S^{\mathrm{abs}}(\lambda)$ is an isometry on a dense subspace of the Hilbert space; it therefore extends uniquely to an isometry on its $L^2$ closure.

It remains to show that the range is all of the Hilbert space.  The incoming/outgoing construction is reversible: if the incoming coefficient is replaced by $S^{\mathrm{abs}}(\lambda)\Phi$, then applying the same Green identity to the solution with incoming and outgoing roles interchanged ($\lambda\rightarrow -\lambda$) gives an incoming coefficient $\Phi$.  Equivalently, for $\operatorname{Im}\lambda>0$ the outgoing resolvent construction gives the adjoint relation
\begin{align}\label{Sadjointanalytic}
        S^{\mathrm{abs}}(\lambda)^*=S^{\mathrm{abs}}(\overline{\lambda})^{-1},
\end{align}
by applying Green's identity to $E_{\lambda,M}(\Phi)$ and $E_{\overline{\lambda},M}(\Psi)$ and then passing to the boundary.  Taking nontangential limits to the real axis in \eqref{Sadjointanalytic} gives $S^{\mathrm{abs}}(\lambda)^*=S^{\mathrm{abs}}(\lambda)^{-1}$ for $\lambda>0$.  Hence the isometric extension is onto, and $S^{\mathrm{abs}}(\lambda)$ is unitary. Since $C_\pm$ are unitary phase factors, the phase-normalized matrix $\mathcal{S}^t_M(\lambda)=C_+S^{\mathrm{abs}}(\lambda)C_-^{-1}$ is unitary as well.
\end{proof}
\begin{proof}[Proof of Theorem \ref{lengths}]
The result follows by combining Proposition~\ref{scatteringbase} with the conjugation formula \eqref{phase-abs-relation} and Proposition~\ref{uniS}. The outgoing correction in Proposition~\ref{genrepA} is represented near $x=0$ by a term of the form
$$
e^{i\lambda/x}x^{(n-1)/2}\Theta(\Phi),
$$
where $\Theta(\Phi)$ is smooth up to the boundary and depends smoothly on the boundary datum $\Phi$. Hence its leading coefficient
$$
\Theta_0(\Phi):=\Theta(\Phi)\big|_{x=0}=C_+A^{\mathrm{abs}}_{M,\lambda}(\Phi).
$$
By Proposition~14 of \cite{melbook2}, the outgoing solution operator sends compactly supported smooth data to a solution whose outgoing leading coefficient is obtained by a smooth boundary kernel, so the map defining $A^{\mathrm{abs}}_\lambda$ is smoothing. Let
\[
\alpha_M(\lambda)
:=\mathcal{S}^t_M(\lambda)-\mathcal{S}^t_{\widetilde M_0}(\lambda)
=C_+A^{\mathrm{abs}}_{M,\lambda}C_-^{-1}.
\]
Since $C_\pm$ are diagonal unitary maps and $A^{\mathrm{abs}}_{M,\lambda}$ is smoothing, $\alpha_M(\lambda)$ is smoothing. Under assumption \eqref{closeness2}, the same argument gives a map $H^{-s}(N;\Lambda^pT^*N)\rightarrow L^2(N;\Lambda^pT^*N)$ for $s<n_0-1$, as in \cite{JSB}.
\end{proof}
The proof of Corollary~\ref{lengthscor} then follows immediately from \eqref{alphatilde}.

In the low-energy estimates below we sometimes write $A_\lambda$ for $A^{\mathrm{abs}}_\lambda$ to lighten notation.
For completeness, we state the limiting absorption principle
\begin{theo}\label{localresolvent}
Let $\lambda\in \mathbb{R}$, $\lambda\neq 0$. Then 
\begin{align}
x^{\alpha} R_{\lambda+i0} x^{\alpha}=\lim\limits_{\epsilon \rightarrow 0^+}x^{\alpha} R_{(\lambda+i\epsilon)} x^{\alpha}
\end{align}
is a continuous map from $L^2_{comp}(M;\Lambda^{p}T^*M)$ to $L^2_{\loc}( M;\Lambda^{p}T^* M)$ for $\alpha>\frac{1}{2}$.
\end{theo}
\begin{proof}
The proof follows that of \cite{melbook2}, Proposition~14, for differential forms.  The uniqueness of solutions to the Helmholtz equation follows from Proposition~\ref{prop:no-positive-eigenforms}.
\end{proof} 

\section{Results}
\subsection{Statement of the main theorem}
Let $(M,g)$ be an $n$-dimensional asymptotically conic manifold whose cross-section is a closed smooth Riemannian manifold, $(N,h_0)$. The manifold is  the interior of a smooth compact manifold $\overline{M}$ with boundary $\partial{\overline{M}}=N$ equipped with a complete smooth metric $g$ such that the metric on the end $Z=(0,\epsilon_1)\times N$ for some $\epsilon_1\in (0,1]$ takes the form
\begin{align}\label{conicmetric}
g=\frac{dx^2}{x^4} + \frac{h (x)}{x^2},
\end{align} 
where $h(x)$ is a smooth family of metrics on $N$ with $h(0)=h_0$.  We let $\widetilde{M}_0$ be the exact manifold with metric 
\[
g_0=\frac{dx^2}{x^4} + \frac{h_0}{x^2}
\] 
(in other words with constant metric $h_0$ on $N$ for all $x$). The manifold $\widetilde{M}$ is a compact perturbation of $\widetilde{M}_0$, that is outside a compact set $\widetilde{M}$ has an exact end $Z=(0,\epsilon_1)\times N$ with  metric $g_0$. Here it is understood that at the possible singularity   $x=\infty$ (resp. $x=1/r$, $r=0$), the Hodge-Laplacian on $\widetilde{M}_0$ is understood to be its {relative} extension. Note that in the case when $N$ is the sphere, there is no singularity at $x=\infty$ and the manifold is called asymptotically Euclidean.   We assume that on the subset $Z$, the metrics on $M$ and $\widetilde{M}$  satisfy
\begin{align}\label{closeness}
h_0-h(x)=O(x^{\infty}),
\end{align}
or that 
\begin{align}\label{closeness2}
h_0-h(x)=O(x^{n_0}), \quad n_0\geq 3
\end{align}
as $x\to 0$.

Throughout this paper all differential forms are complex-valued, so $\Lambda^pT^*M$ stands for $\Lambda^p_{\mathbb{C}}T^*M$; we suppress the subscript $\mathbb{C}$ where no confusion can arise (similarly for $\widetilde{M}_0$ and $\widetilde{M}$).  The Hodge Laplacian on $p$-forms is defined by $\Delta_p=d\delta+\delta d$ with spectrum contained in $[0,\infty)$.  The corresponding resolvent, $R_{\lambda}=(\Delta_p-\lambda^2)^{-1}$ is a well-defined operator for $\Im(\lambda)> 0$ on $L^2(M,dg)$. By Theorem~1 in \cite{melbook2}
the resolvent $R_{\lambda}$ has a meromorphic extension to $\mathbb{C}$ in odd dimensions and $\mathbb{C}\setminus [0,\infty)$ in even dimensions as a family of bounded operators from $H^s_{\comp}(M;\Lambda^pT^*M)\rightarrow H^{s+2}_{\loc}(M;\Lambda^pT^*M)$ with finite rank negative Laurent coefficients.

The absence of non-zero $L^2$ eigenforms at positive energy is automatic in the present geometric setting, by the following recent result.
\begin{prop}\label{prop:no-positive-eigenforms}[\cite[Proposition~4.6]{ACCL}]
Let $(X,g_X)$ be a smooth complete asymptotically conic Riemannian manifold.  For every degree $p$ and every $\mu\in\mathbb R\setminus\{0\}$, there is no non-zero form
\[
u\in\operatorname{dom}(\Delta_p)\subset L^2(X;\Lambda^pT^*X)
\qquad\text{such that}\qquad
\Delta_p u=\mu u.
\]
\end{prop}
Consequently, neither $M$ nor the smooth compact perturbation $\widetilde M$ has positive embedded eigenvalues for the Hodge Laplacian in any degree.

To introduce the main theorem, we let $f(x)$ be a smooth function of $x$ and let $\eta$ be a $p$-form on $N$. In particular, the space of $L^2$-integrable $p$-forms on $M$ (and also on $Z$ with metric $g_0$) near the boundary we will show is given by forms of type 
\begin{equation}\label{psi}
\psi(x,y)=f(x) x^{\frac{n-1}{2} -p} \, \eta_1(y)+g(x) x^{\frac{n-1}{2} -(p-1)-2}\, dx\wedge \eta_2(y)
\end{equation}
where $f, g\in C^{\infty}(0,\epsilon_1),$ $\eta_1(y)\in \Lambda^p T^* N$ and $\eta_2(y)\in \Lambda^{p-1} T^* N$ on the submanifold $N$. Note that on the exact manifold ${\widetilde{M}_0}$
\begin{align}
 |x^{-p} \, \eta_1(y)|_{\widetilde{M}_0} = |\eta_1(y)|_N \quad \mathrm{and}\quad |x^{-p-1}\, dx\wedge\eta_2(y)|_{\widetilde{M}_0} = |\eta_2(y)|_N 
 \end{align}
and on $M$ these also hold asymptotically as $x\to 0$ by the assumptions on $h$.

Based on this decomposition, we now define a projection mapping on $Z$ 
\begin{align}
{\Pi}_p:  L^2(Z;\Lambda^p T^*Z; \vol_Z) \rightarrow L^2\left(\,\left((0,\epsilon_1);  \Lambda^p T^*N\right) \oplus \left((0,\epsilon_1);\Lambda^{p-1}T^* N\right), x^{-2} \vol_{N}\, dx \right) 
\end{align}	
such that
\begin{equation}
{\Pi}_p \psi = \left(\begin{matrix}f\, \eta_1 \\ g \, \eta_2\end{matrix} \right)
\end{equation}
where $\psi$  is decomposed as in \eqref{psi}.

We define the phase-normalized tangential scattering matrix as follows. 
\begin{def1}\label{def:phase-scattering}
 Suppose ${\lambda>0}$ and $u\in C^{\infty}(M,\Lambda^{p}T^*M)$ satisfies
\begin{align}
(\Delta_p-\lambda^2)u=0, \qquad  \delta u=0, \qquad \text{and} \quad \Pi_pu=e^{-i\frac{\lambda}{x}}\Phi+e^{i\frac{\lambda}{x}}\Psi+O(x) \quad \text{as} \quad x\rightarrow 0.
\end{align}
Then the tangential scattering matrix is defined by
\begin{align}
\mathcal{S}^t(\lambda): C^{\infty}(N,\Lambda^pT^*N)\rightarrow C^{\infty}(N,\Lambda^pT^*N); \quad \mathcal{S}^t(\lambda):\Phi\mapsto \Psi.
\end{align}
Here $\Phi,\Psi\in C^{\infty}(N,\Lambda^pT^*N)$ denote the tangential components of the pair $\Pi_p u$ determined through the co-closedness condition $\delta u=0$.
\end{def1}
The full scattering matrix, including the normal part, can be obtained by Hodge duality. In particular,  for the exact model $\widetilde{M}_0$  the tangential scattering model matrix is diagonal  
\begin{align}
\mathcal{S}_{\widetilde{M}_0}^t(\lambda)=\begin{pmatrix}
\mathcal{S}_{\widetilde{M}_0,1}^t & 0 \\
0 & \mathcal{S}_{\widetilde{M}_0,2}^t
\end{pmatrix}
\end{align}
where $\mathcal{S}^t_1$ acts on co-closed forms on $N$ and $\mathcal{S}^t_2$ acts on {nonharmonic} closed forms on $N$. 
The following theorem describes the scattering matrix and is the main result of this paper.
\begin{theo}\label{lengths}
Suppose that the metric on $M$ satisfies \eqref{closeness}.  Then, for the exact model $\widetilde{M}_0$ the tangential scattering matrix on $N$ has the two components {on nonharmonic angular data}
\begin{align}
\mathcal{S}^t_{1,\widetilde M_0}(\lambda)&=-i e^{-i\pi\sqrt{\left(\frac{n}{2}-p-1\right)^2+\Delta_N}},
&
\mathcal{S}^t_{2,\widetilde M_0}(\lambda)&=i e^{-i\pi\sqrt{\left(\frac{n}{2}-p\right)^2+\Delta_N}}.
\end{align}
These act on co-closed and closed forms on $N$, respectively. {The harmonic case is given by \eqref{harmonic-phase}.} The tangential scattering matrix on $M$ satisfies
\begin{align}
\mathcal{S}^t_M(\lambda)=\mathcal{S}^t_{\widetilde M_0}(\lambda)+\alpha_M(\lambda),
\end{align}
where $\alpha_M(\lambda)$ is a smoothing operator on $C^{\infty}(N;\Lambda^pT^*N)$. Under assumption \eqref{closeness2}, the same conclusion holds with $\alpha_M(\lambda)$ replaced by an operator
\[
\alpha_M(\lambda):H^{-s}(N;\Lambda^pT^*N)\longrightarrow L^2(N;\Lambda^pT^*N),
\qquad s<n_0-1.
\]
\end{theo}
The two diagonal factors in Theorem~\ref{lengths} arise from the closed/co-closed Hodge decomposition of the tangential data at infinity.  Thus, beyond recovering the scalar phase in degree zero, the theorem exhibits the additional channel structure intrinsic to differential forms.  The proof proceeds through the explicit generalized eigenforms constructed below, rather than through the Fourier integral operator description used in the scalar setting of \cite{MZInvent}. 
\begin{cor}[Fourier integral structure]\label{cor:fio}
Assume \eqref{closeness}, and let $\mathcal G^t$ be the Hamiltonian flow of $|\xi|_{h_0}$ on $T^*N\setminus 0$. For each fixed $\lambda>0$, the tangential scattering matrix $\mathcal S^t_M(\lambda)$ is an order-zero Fourier integral operator associated with the graph of $\mathcal G^\pi$, with the canonical-relation convention of $e^{-it\sqrt{\Delta_N}}$. After factoring out the scalar half-wave symbol, including its half-density and Maslov factors, its principal symbol is
\[
\tau_\pi\bigl(-i\pi_{\mathrm{cc}}+i\pi_{\mathrm{cl}}\bigr).
\]
Here $\pi_{\mathrm{cc}}$ and $\pi_{\mathrm{cl}}$ are the principal symbols of the projections onto the nonharmonic co-closed and closed forms, evaluated at the initial covector, and $\tau_\pi$ is Levi-Civita parallel transport on $\Lambda^pT^*N$ along the corresponding geodesic of length $\pi$.
\end{cor}
\begin{proof}
Let $Q$ be the inverse of $\Delta_N$ on the orthogonal complement of its kernel, extended by zero on the kernel. Then $Q$ is pseudodifferential of order $-2$, so the Hodge projections
\[
\Pi_{\mathrm{cc}}=\delta_N d_N Q,
\qquad
\Pi_{\mathrm{cl}}=d_N\delta_N Q
\]
are pseudodifferential of order zero. Write $A_1,A_2$ for the two square-root operators in Theorem~\ref{lengths}. Each has scalar principal symbol $|\xi|_{h_0}\Id$, so the standard half-wave construction gives order-zero Fourier integral operators $e^{-i\pi A_j}$ associated with the graph of $\mathcal G^\pi$. Theorem~\ref{lengths} gives
\[
\mathcal S^t_M(\lambda)
=-i e^{-i\pi A_1}\Pi_{\mathrm{cc}}
 +i e^{-i\pi A_2}\Pi_{\mathrm{cl}}
 +R_{\mathrm{harm}}+\alpha_M(\lambda),
\]
where $R_{\mathrm{harm}}$ is the finite-rank smoothing contribution from \eqref{harmonic-phase}, and $\alpha_M(\lambda)$ is smoothing under \eqref{closeness}. Composition with the Hodge projections therefore proves the Fourier integral assertion. Moreover, $A_j-\sqrt{\Delta_N}$ has order $-1$, so the constant shifts do not change the subprincipal transport equation. For the Hodge Laplacian this equation gives Levi-Civita parallel transport \cite[Sections~3 and~5]{JSt}. Composing its symbol with those of the Hodge projections and including the phases $-i$ and $i$ gives the stated principal symbol.
\end{proof}

\begin{cor}\label{lengthscor}
Assume that the metrics on $M$ and $\widetilde M$ satisfy \eqref{closeness}. Then, for $\lambda\in\mathbb R^+$, the difference
\[
\widetilde\alpha(\lambda):=\mathcal{S}^t_M(\lambda)-\mathcal{S}^t_{\widetilde M}(\lambda)
\]
is a smoothing operator on $C^{\infty}(N;\Lambda^pT^*N)$. Under assumption \eqref{closeness2}, $\widetilde\alpha(\lambda)$ instead extends to a bounded operator
\[
\widetilde\alpha(\lambda):H^{-s}(N;\Lambda^pT^*N)\longrightarrow L^2(N;\Lambda^pT^*N),
\qquad s<n_0-1.
\]
\end{cor}

\subsection{Scattering matrix and spectral representation}

For functions $f,g:V \to W$ valued in a locally convex topological vector space $W$ and $h:  V \to \R$, we write $f = g + O_W(h)$ if for every continuous seminorm $q$ on $W$ there is a constant $C_q$ such that $q(f(\lambda)-g(\lambda)) \leq C_q |h(\lambda)|$ for all $\lambda \in V$.
 
\begin{lem}\label{grossboundsE}
For any $D>0$ and any basis vector $\Phi_\sigma$ for $L^2(N,\Lambda^pT^*N)$ with eigenvalue $\sigma$ as in Lemma~\ref{basis}, we have that for $|\lambda|<D$
\begin{align}\label{grossboundsJ}
\tilde{j}_{\lambda}(\Phi_{\sigma})=O_{C^{\infty}(\widetilde{M})} \left(\max\left\{\frac{\lambda^{\ell_{p}(\sigma)+\frac{1}{2}}}{\Gamma(\ell_{p}(\sigma)
+1)},  \frac{\lambda^{\ell_{p-1}(\sigma)-\frac{1}{2}} }{\Gamma(\ell_{p-1}(\sigma)+1)},\frac{\lambda^{\ell_{p-1}(\sigma)-\frac{1}{2}} }{\Gamma(\ell_{p-1}(\sigma))} \right\} \right).
\end{align} 
Moreover for all $|\lambda|<D$ and $\mathrm{Im}\lambda\geq 0$ we have that
\begin{align}\label{grossboundsB}
E_{\lambda,M}(\Phi_\sigma)|_Z= O_{C^{\infty}(Z)}\left(\max\left\{\frac{\lambda^{\ell_{p}(\sigma)-\frac{3}{2}}}{\Gamma(\ell_{p}(\sigma)
+1)}, \frac{\lambda^{\ell_{p-1}(\sigma)-\frac{5}{2}} }{\Gamma(\ell_{p-1}(\sigma)+1)},\frac{\lambda^{\ell_{p-1}(\sigma)-\frac{5}{2}}}{\Gamma(\ell_{p-1}(\sigma))} \right\} \right).
\end{align}
\end{lem}
\begin{proof}
Using \eqref{eq11}, \eqref{derivativebessel} and the upper bound \eqref{besselineq} we obtain that 
\[
\begin{split}\label{eigenfupper}
|\tilde{j}_{\lambda}(\Phi_{\sigma})|\leq \; & C\sqrt{\lambda} \left[ \frac{1}{\sqrt x} J_{\ell_{p}(\sigma)} + \frac{\sqrt x}{\lambda} J_{\ell_{p-1}(\sigma)}  + \frac{1}{\sqrt{x}} J_{\ell_{p-1}(\sigma)-1} \right]\\
\leq \; &C \left[\frac{\lambda^{\ell_{p}(\sigma)+\frac{1}{2}}}{x^{\ell_{p}(\sigma)+\frac{1}{2}} \,\Gamma(\ell_{p}(\sigma)+1)} +  \frac{\lambda^{\ell_{p-1}(\sigma)-\frac{1}{2}}}{x^{\ell_{p-1}(\sigma) -\frac{1}{2}} \,\Gamma(\ell_{p-1}(\sigma)+1)} \right.\\
& \quad \left.+  \frac{\lambda^{\ell_{p-1}(\sigma)-\frac{1}{2}}}{x^{\ell_{p-1}(\sigma) -\frac{1}{2}} \,\Gamma(\ell_{p-1}(\sigma))}  \right].
\end{split}
\]
Therefore, the family $\tilde{j}_{\lambda}(\Phi_{\sigma})$ is bounded in $L^2_{\loc}( M;\Lambda^{p}T^*M)$ and \eqref{grossboundsJ} follows. $\Delta\tilde{j}_{\lambda}(\Phi_{\sigma})=\lambda^2\tilde{j}_{\lambda}(\Phi_{\sigma})$, this shows that the family is bounded in $H^s_{loc}( M;\Lambda^{p}T^*M)$ for any $s\in 2\mathbb{N}$. The resolvent $R_{\lambda}$ maps $H^s_{comp}(M;\Lambda^{p}T^* M)\rightarrow H_{loc}^{s+2}(M;\Lambda^{p}T^*M)$, and since the resolvent has a singularity of order at most two at zero and is analytic near the real line and in the upper half plane, it follows that $E_{\lambda,M}(\Phi_\sigma)|_Z$ has an additional factor of order $O(\lambda^{-2})$ as $\lambda \to 0$ due to $R_{\lambda}(\Delta_M-\lambda^2)(\chi \tilde{j}_{\lambda}(\Phi))$, giving the stated bound.
\end{proof}

The following lemma controls the low-energy behaviour of the scattering amplitude.
\begin{lem}\label{scatteringAbounds}
For any $C^{\frac{q}{2}}$ semi-norm on $Z$, there exists a constant $\delta_1\in (0,\epsilon_1)$ such that for any $\lambda\in\mathbb{C}$ with $0<|\lambda|\leq \lambda_0$, with $\lambda_0$ small a positive the following bounds hold:\\
If $\Phi_{\sigma}$ and $\Phi_{\nu}$ are of the type $\varphi_{\sigma}^{(p)}$ and $\varphi_{\nu}^{(p)}$ respectively and $\ell_p(\nu)\neq 0$:
\begin{align}\label{Abound1}
&|\langle A_{\lambda}\Phi_{\sigma},\Phi_{\nu}\rangle|=O_q\left(\frac{\langle \ell_p(\sigma)\rangle^{q}
\langle \ell_p(\nu)\rangle^{q}\delta_1^{-\ell_p(\nu)-\ell_p(\sigma)}}{\Gamma(\ell_{p}(\sigma)
+1)\Gamma(\ell_p(\nu))}\lambda^{\ell_{p}(\sigma)+\ell_{p}(\nu)-2}\right).
\end{align}
If $\ell_p(\nu)=0$ then we have
\begin{align}
|\langle A_{\lambda}\Phi_{\sigma},\Phi_{\nu}\rangle|=O_q\left(\frac{\langle \ell_p(\sigma)\rangle^{q}
}{\Gamma(\ell_{p}(\sigma)
+1)|\log\lambda|}\lambda^{\ell_{p}(\sigma)-2}\delta_1^{-\ell_p(\sigma)}\right).
\end{align} 
However, if $\Phi_{\sigma}$ and $\Phi_{\nu}$ are of the type $\xi_{\sigma}^{(p)}$ and $\xi_{\nu}^{(p)}$ with $\ell_{p-1}(\nu)>1$ we have: 
\begin{align}\label{Abound2}
&|\langle A_{\lambda}\Phi_{\sigma},\Phi_{\nu}\rangle|=O_q\left(\frac{\langle \ell_p(\sigma)\rangle^{q}
\langle \ell_p(\nu)\rangle^{q}\delta_1^{-\ell_{p-1}(\nu)-\ell_{p-1}(\sigma)}}{\Gamma(\ell_{p-1}(\sigma))\Gamma(\ell_{p-1}(\nu)-1)
}\lambda^{\ell_{p-1}(\sigma)+\ell_{p-1}(\nu)-4}\right).
\end{align}
Moreover, if $\ell_{p-1}(\nu)=1$  we have that 
\begin{align}
|\langle A_{\lambda}\Phi_{\sigma},\Phi_{\nu}\rangle|=O_q\left(\frac{\langle \ell_p(\sigma)\rangle^{q}
}{\Gamma(\ell_{p-1}(\sigma))|\log\lambda|
}\lambda^{\ell_{p-1}(\sigma)-3}\delta_1^{-\ell_{p-1}(\sigma)}\right).
\end{align}
If $\ell_{p-1}(\nu)=0$ we have that there is cancellation and 
\begin{align}
|\langle A_{\lambda}\Phi_{\sigma},\Phi_{\nu}\rangle|=O_q\left(\frac{\langle \ell_p(\sigma)\rangle^{q}}
{\Gamma(\ell_{p-1}(\sigma))
}\lambda^{\ell_{p-1}(\sigma)-2}\delta_1^{-\ell_{p-1}(\sigma)}\right).
\end{align}
When $\Phi_{\sigma}$ and $\Phi_{\nu}$ are of the type $\xi_{\sigma}^{(p)}$ and $\xi_{\nu}^{(p)}$ with $\ell_{p-1}(\nu)\in (0,1)$ we have: 
\begin{align}\label{Abound3}
&|\langle A_{\lambda}\Phi_{\sigma},\Phi_{\nu}\rangle|=O_q\left(\frac{\langle \ell_p(\sigma)\rangle^{q}
\langle \ell_p(\nu)\rangle^{q}\delta_1^{-\ell_{p-1}(\nu)-\ell_{p-1}(\sigma)}}{\Gamma(\ell_{p-1}(\sigma))\Gamma(\ell_{p-1}(\nu))
}\lambda^{\ell_{p-1}(\sigma)+\ell_{p-1}(\nu)-3}\right).
\end{align}
Finally, we record the mixed terms, in which $\Phi_\sigma$ and $\Phi_\nu$ are of different Hodge type.  If $\Phi_{\sigma}$ is of the type $\varphi^{(p)}_{\sigma}$ and $\Phi_{\nu}$ is of the type $\xi^{(p)}_{\nu}$, then
\begin{align}\label{Amixed1}
|\langle A_{\lambda}\Phi_{\sigma},\Phi_{\nu}\rangle|=O_q\left(\frac{\langle \ell_p(\sigma)\rangle^{q}\langle \ell_p(\nu)\rangle^{q}\delta_1^{-\ell_{p}(\sigma)-\ell_{p-1}(\nu)}}{\Gamma(\ell_{p}(\sigma)+1)\,\Gamma(\ell_{p-1}(\nu)-1)}\lambda^{\ell_{p}(\sigma)+\ell_{p-1}(\nu)-3}\right)
\qquad \ell_{p-1}(\nu)>1,
\end{align}
\begin{align}\label{Amixed1log}
|\langle A_{\lambda}\Phi_{\sigma},\Phi_{\nu}\rangle|=O_q\left(\frac{\langle \ell_p(\sigma)\rangle^{q}\delta_1^{-\ell_{p}(\sigma)}}{\Gamma(\ell_{p}(\sigma)+1)\,|\log\lambda|}\lambda^{\ell_{p}(\sigma)-2}\right)
\qquad \ell_{p-1}(\nu)=1,
\end{align}
\begin{align}\label{Amixed1low}
|\langle A_{\lambda}\Phi_{\sigma},\Phi_{\nu}\rangle|=O_q\left(\frac{\langle \ell_p(\sigma)\rangle^{q}\langle \ell_p(\nu)\rangle^{q}\delta_1^{-\ell_{p}(\sigma)-\ell_{p-1}(\nu)}}{\Gamma(\ell_{p}(\sigma)+1)\,\Gamma(\ell_{p-1}(\nu))}\lambda^{\ell_{p}(\sigma)+\ell_{p-1}(\nu)-2}\right)
\qquad \ell_{p-1}(\nu)\in(0,1).
\end{align}
If $\Phi_{\sigma}$ is of the type $\xi^{(p)}_{\sigma}$ and $\Phi_{\nu}$ is of the type $\varphi^{(p)}_{\nu}$, then
\begin{align}\label{Amixed2}
|\langle A_{\lambda}\Phi_{\sigma},\Phi_{\nu}\rangle|=O_q\left(\frac{\langle \ell_p(\sigma)\rangle^{q}\langle \ell_p(\nu)\rangle^{q}\delta_1^{-\ell_{p-1}(\sigma)-\ell_{p}(\nu)}}{\Gamma(\ell_{p-1}(\sigma))\,\Gamma(\ell_{p}(\nu))}\lambda^{\ell_{p-1}(\sigma)+\ell_{p}(\nu)-3}\right)
\qquad \ell_{p}(\nu)\neq0,
\end{align}
\begin{align}\label{Amixed2log}
|\langle A_{\lambda}\Phi_{\sigma},\Phi_{\nu}\rangle|=O_q\left(\frac{\langle \ell_p(\sigma)\rangle^{q}\delta_1^{-\ell_{p-1}(\sigma)}}{\Gamma(\ell_{p-1}(\sigma))\,|\log\lambda|}\lambda^{\ell_{p-1}(\sigma)-3}\right)
\qquad \ell_{p}(\nu)=0.
\end{align}
\end{lem}
\begin{proof}
The proof follows by combining the previous Lemma \ref{grossboundsE} and the asymptotic \eqref{eqn:HankelSmallAsympt}, and it is similar to that of Lemma 2.10 in \cite{OS}. We start with the bound \eqref{grossboundsB} from Lemma \ref{grossboundsE}
\begin{align}
\Pi_pE_{\lambda,M}(\Phi_{\sigma})|_Z= O_{C^{\infty}(Z)}\left(\max\left\{\frac{\lambda^{\ell_{p}(\sigma)-\frac{3}{2}}}{\Gamma(\ell_{p}(\sigma)
+1)}, \frac{\lambda^{\ell_{p-1}(\sigma)-\frac{5}{2}} }{\Gamma(\ell_{p-1}(\sigma)+1)},\frac{\lambda^{\ell_{p-1}(\sigma)-\frac{5}{2}}}{\Gamma(\ell_{p-1}(\sigma))} \right\}  \right).
\end{align}
We then recall that 
\begin{align*}
E_{\lambda,M}(\Phi_{\sigma})|_{Z}=\tilde{j}_{\lambda}(\Phi_{\sigma})+\tilde{h}_{\lambda}^{(1)}(A_{\lambda}\Phi_{\sigma})+O(x^{\frac{n+1}{2}}).
\end{align*}
As in the proof of the previous lemma, the leading order terms in the expansion are from   $\tilde{h}^{(1)}(A_{\lambda}\Phi_{\sigma})$, because the $\tilde{j}_{\lambda}(\Phi_{\sigma})$ terms are all of lower order as $\lambda\rightarrow 0$. In the first case where $\Phi_{\sigma}$ and $\Phi_{\nu}$ are of the type $\varphi_{\sigma}^{(p)}$ and $\varphi_{\nu}^{(p)}$ respectively, with $\sigma, \nu \neq 0$, Definition \ref{def:hankel-solutions} and \eqref{eq11} give for $x$ with $\delta_1<|x|<\epsilon_1$   
\begin{align*}
&\left|\langle A_{\lambda}\Phi_{\sigma},\Phi_{\nu}\rangle\right| \; \left|\sqrt{\frac{\lambda}{x}}H^{(2)}_{\ell_{p}(\nu)}\left(\frac{\lambda}{x}\right)\right| \\
&=
O_{C^{\infty}(Z)}\left(\max\left\{\frac{\lambda^{\ell_{p}(\sigma)-\frac{3}{2}}}{\Gamma(\ell_{p}(\sigma)
+1)},   \frac{\lambda^{\ell_{p-1}(\sigma)-\frac{5}{2}} }{\Gamma(\ell_{p-1}(\sigma)+1)},\frac{\lambda^{\ell_{p-1}(\sigma)-\frac{5}{2}}}{\Gamma(\ell_{p-1}(\sigma))} \right\} \right),
\end{align*}
which for $\lambda$ near 0 by \eqref{eqn:HankelSmallAsympt} results in   
\begin{align}
&|\langle A_{\lambda}\Phi_{\sigma},\Phi_{\nu}\rangle|\\
& =O \left(\max\left\{\frac{\lambda^{\ell_{p}(\sigma)-\frac{3}{2}}}{\Gamma(\ell_{p}(\sigma)
+1)},   \frac{\lambda^{\ell_{p-1}(\sigma)-\frac{5}{2}} }{\Gamma(\ell_{p-1}(\sigma)+1)},\frac{\lambda^{\ell_{p-1}(\sigma)-\frac{5}{2}}}{\Gamma(\ell_{p-1}(\sigma))} \right\}  \frac{\lambda^{\ell_{p}(\nu)-\frac{1}{2}}}{\Gamma(\ell_{p}(\nu))} \langle \ell_p(\sigma)\rangle^{q}
\langle \ell_p(\nu)\rangle^{q}\right),
\end{align}
provided $\ell_p(\nu)\neq 0$. 
 If $\ell_p(\nu)=0$ then 
\begin{align}
|\langle A_{\lambda}\Phi_{\sigma},\Phi_{\nu}\rangle|& =O_q \left(\max\left\{\frac{\lambda^{\ell_{p}(\sigma)-\frac{3}{2}}}{\Gamma(\ell_{p}(\sigma)
+1)},   \frac{\lambda^{\ell_{p-1}(\sigma)-\frac{5}{2}} }{\Gamma(\ell_{p-1}(\sigma)+1)},\frac{\lambda^{\ell_{p-1}(\sigma)-\frac{5}{2}}}{\Gamma(\ell_{p-1}(\sigma))} \right\}  \frac{\lambda^{-\frac 12}}{|\log\lambda|} \langle \ell_p(\sigma)\rangle^{q}\right) .
\end{align} 
Otherwise  in the second case, Definition \ref{def:hankel-solutions} and \eqref{eq11} give 
\begin{align}\label{bad}
|\langle A_{\lambda}\Phi_{\sigma},\Phi_{\nu}\rangle| &\left| \sqrt{\frac{\lambda}{x}} H^{(2)}_{\ell_{p-1}(\nu)-1}\left(\frac{\lambda}{x}\right)   \right| \\
&=O_{C^{\infty}(Z)}\left(\max\left\{\frac{\lambda^{\ell_{p}(\sigma)-\frac{3}{2}}}{\Gamma(\ell_{p}(\sigma) \nonumber
+1)},   \frac{\lambda^{\ell_{p-1}(\sigma)-\frac{5}{2}} }{\Gamma(\ell_{p-1}(\sigma)+1)},\frac{\lambda^{\ell_{p-1}(\sigma)-\frac{5}{2}}}{\Gamma(\ell_{p-1}(\sigma))} \right\} \right),\nonumber
\end{align}
which gives, when $\ell_{p-1}(\nu)\neq 1$, $\ell_{p-1}(\nu)>1$
\begin{align}
&|\langle A_{\lambda}\Phi_{\sigma},\Phi_{\nu}\rangle|=O_q \left(\max\left\{\frac{\lambda^{\ell_{p}(\sigma)-\frac{3}{2}}}{\Gamma(\ell_{p}(\sigma)
+1)},   \frac{\lambda^{\ell_{p-1}(\sigma)-\frac{5}{2}} }{\Gamma(\ell_{p-1}(\sigma)+1)},\frac{\lambda^{\ell_{p-1}(\sigma)-\frac{5}{2}}}{\Gamma(\ell_{p-1}(\sigma))} \right\}  \frac{\lambda^{\ell_{p-1}(\nu)-\frac{3}{2}}}{\Gamma(\ell_{p-1}(\nu)-1)} \right)  \notag\\ 
\end{align}
and similarly for the other terms in the expansion, since this is the largest term in the expansion for small $|\lambda|$. For the terms in the region $\ell_{p-1}(\nu)\in (0,1)$ the order $\ell_{p-1}(\nu)-1\in(-1,0)$ is negative, so \eqref{eqn:HankelSmallAsympt} no longer governs the leading behaviour of $H^{(2)}_{\ell_{p-1}(\nu)-1}$. We therefore evaluate the derivative in the $B_{p+1}$ computation \eqref{eq11} using \eqref{derivativebesselpositive} in place of \eqref{derivativebessel}, so that the pairing is taken against $H^{(2)}_{\ell_{p-1}(\nu)}$, whose order $\ell_{p-1}(\nu)\in(0,1)$ is positive. Applying \eqref{eqn:HankelSmallAsympt} to this factor replaces $\Gamma(\ell_{p-1}(\nu)-1)$ by $\Gamma(\ell_{p-1}(\nu))$ and raises the power of $\lambda$ by one relative to \eqref{Abound2}, which gives \eqref{Abound3}. However, when $\ell_{p-1}(\nu)=1$ we have that 
\begin{align}
|\langle A_{\lambda}\Phi_{\sigma},\Phi_{\nu}\rangle|\left|\sqrt{\frac{\lambda}{x}} H_{\ell_{ p-1}(\nu)-1}^{(2)}\left(\frac{\lambda}{x}\right) \right|=O\left(\frac{1}{\Gamma(\ell_{p-1}(\sigma)
)}\lambda^{\ell_{p-1}(\sigma)-\frac{5}{2}}\langle \ell_p(\sigma)\rangle^{q}\right)
\end{align}
and therefore
\begin{align}
|\langle A_{\lambda}\Phi_{\sigma},\Phi_{\nu}\rangle|=O_q\left(\frac{1}{\Gamma(\ell_{p-1}(\sigma))|\log\lambda|
}\lambda^{\ell_{p-1}(\sigma)-3}\langle \ell_p(\sigma)\rangle^{q}\right).
\end{align}
Now $\ell_{p-1}(\nu)=0$ if and only if $\nu=0$ and $n/2=p$. Then 
\begin{align}
&B_{p+1}\left(\frac{1}{\lambda\sqrt{x}}H_0^{(1)}\left(\frac{\lambda}{x}\right) \left(\begin{matrix}0 \\ \xi_{\nu}^{(p)}\end{matrix}\right)\right)=\\&\left(\frac{n}{2}-p\right)\frac{\sqrt{x}}{\lambda}H_0^{(1)}\left(\frac{\lambda}{x}\right)\left(\begin{matrix}\xi_{\nu}^{(p)} \\ 0\end{matrix}\right)-\frac{1}{\sqrt{x}}H_{1}^{(1)}\left(\frac{\lambda}{x}\right)\left(\begin{matrix}\xi_{\nu}^{(p)} \\ 0\end{matrix}\right).
\end{align}
The first term is 0 and then the small argument expansion gives that 
$$-\frac{1}{\sqrt{x}}H_{1}^{(1)}\left(\frac{\lambda}{x}\right)=\frac{2i}{\pi}\frac{\sqrt{x}}{\lambda}+O\left(\lambda|\log\lambda|\right)$$
Using the equality  \eqref{bad}, this gives the desired result.
The cross terms follow similarly and their computation is omitted. 
\end{proof}
We remark here that in the case $\frac{n}{2}-p\leq 2$ the scattering amplitude singularity as $\lambda\rightarrow 0$ may cause resonance states. A similar remark was made in \cite{GS}. Therefore for Stone's theorem and its corollaries, we exclude $n=2p+j, j=0,1,2,3,4$ with $p\leq \frac{n}{2}$ which sit in this range.

A version of Stone's formula for differential forms is the following. Let $\{\Phi_{\rho}\}$ be an orthonormal basis of $L^2(N;\Lambda^pT^*N)$ consisting of the eigenforms of Lemma~\ref{basis}, indexed with multiplicity, and write $\Delta_N\Phi_{\rho}=\sigma_{\rho}\Phi_{\rho}$. 

\begin{theo}[Stone formula for co-closed $p$-forms, cf. \cite{OS}]\label{stone1}
Assume that $n\neq 2p+j, j=0,1,2,3,4$ with $p\leq \frac{n}{2}$. Let $\lambda>0$ and let $f\in C_0^\infty(M;\Lambda^pT^*M)$ satisfy $\delta f=0$, and write $E_\lambda=E_{\lambda,M}$ for the generalized eigenfunctions on $M$.  Then
\begin{equation}\label{eq:stone-resolvent-jump}
(R_\lambda-R_{-\lambda})f
=
\frac{\rmi}{2\lambda}
\sum_{\rho}E_\lambda(\Phi_\rho)
\langle f,E_\lambda(\Phi_\rho)\rangle_{L^2(M)} ,
\end{equation}
where the series converges in $C^\infty_{\loc}(M;\Lambda^pT^*M)$.  Here $R_\lambda=(\Delta_p-\lambda^2)^{-1}$ denotes the outgoing resolvent and $R_{-\lambda}$ denotes the incoming resolvent.

Equivalently, if $dB_\lambda$ denotes the absolutely continuous spectral measure of $\Delta_p^{1/2}$ on $(0,\infty)$, then for all $f,g\in C_0^\infty(M;\Lambda^pT^*M)$ with $\delta f=0,\delta g=0$ we have that 
\begin{equation}\label{eq:stone-spectral-measure}
\langle dB_\lambda f,g\rangle
=
\frac{1}{2\pi}\chi_{[0,\infty)}(\lambda)
\sum_\rho
\langle f,E_\lambda(\Phi_\rho)\rangle_{L^2(M)}
\langle E_\lambda(\Phi_\rho),g\rangle_{L^2(M)}\,\der\lambda .
\end{equation}
Consequently, for every bounded Borel function $k:\mathbb R\to\mathbb C$,
\begin{align}\label{StoneFull}
\langle k(\Delta_p)f,g\rangle
&=
 k(0)\sum_{j=1}^N\langle f,u_j\rangle_{L^2(M)}\langle u_j,g\rangle_{L^2(M)} \nonumber \\
&\quad+
\frac{1}{2\pi}\sum_\rho\int_0^\infty
k(\lambda^2)
\langle f,E_\lambda(\Phi_\rho)\rangle_{L^2(M)}
\langle E_\lambda(\Phi_\rho),g\rangle_{L^2(M)}\,\der\lambda .
\end{align}
Here $u_1,\ldots,u_N$ are an $L^2$-orthonormal basis for $\ker\Delta_{p}=\mathcal H^p(M)$.
\end{theo}

\begin{proof}
The proof uses the standard Stone argument, with the normalisation fixed by the radiation coefficients of the generalized eigenfunctions.  Let $f\in C_0^\infty(M;\Lambda^pT^*M)$ with $\delta f=0$ and fix $\lambda>0$. The difference of the resolvents exists  because of the limiting absorption principle, Theorem \ref{localresolvent}.  The outgoing and incoming boundary values satisfy
\begin{equation*}
(\Delta_p-\lambda^2)(R_\lambda-R_{-\lambda})f=0 .
\end{equation*}
Moreover $R_{-\lambda}f$ has incoming leading term at infinity.  Thus, for some
$\Psi\in C^\infty(N;\Lambda^pT^*N)$,
\begin{equation}\label{eq:incoming-leading-term}
\Pi_p R_{-\lambda}f
=
-e^{-\rmi\lambda/x}\Psi+O(x)
\qquad \text{as} \quad x\to 0,
\end{equation}
where we have first projected onto the continuous spectrum of $\Delta_p$; and hence
\begin{equation}\label{eq:jump-asymptotics}
\Pi_p(R_\lambda-R_{-\lambda})f
=
 e^{-\rmi\lambda/x}\Psi
 +e^{\rmi\lambda/x}\widetilde\Psi
 +O(x)
\end{equation}
for some $\widetilde\Psi\in C^\infty(N;\Lambda^pT^*N)$.  By the definition of the Poisson operator, $E_\lambda(\Psi)$ is the unique generalized eigenfunction with incoming coefficient $\Psi$.  Indeed $\delta f=0$ forces $\delta E_\lambda(\Psi)=0$, so only co-closed solutions occur and the incoming coefficient determines the eigenfunction uniquely.  We recall the uniqueness argument, since this is the point at which the radiation condition is used.

Set
\begin{equation*}   
w=(R_\lambda-R_{-\lambda})f-E_\lambda(\Psi).
\end{equation*}
Then $(\Delta_p-\lambda^2)w=0$, $w$ satisfies the outgoing radiation condition, and the incoming coefficient of $w$ vanishes.  Thus $w$ is purely outgoing at infinity.  Let $M_\epsilon=\{x\geq\epsilon\}$.  Green's identity for differential forms gives\begin{equation}\label{eq:green-flux} 
0
=
\langle (\Delta_p-\lambda^2)w,w\rangle_{M_\epsilon}
-
\langle w,(\Delta_p-\lambda^2)w\rangle_{M_\epsilon}
=
b_{\epsilon}(w,w),
\end{equation}
where $b_{\epsilon}(w,w)$ is the boundary pairing in \eqref{bpairing}. If the outgoing coefficient of $w$ is denoted by $a\in C^\infty(N;\Lambda^pT^*N)$, then evaluating the boundary pairing mode by mode and the previous computations of the boundary pairings for the generalized eigenfunctions gives
\begin{equation}\label{eq:flux-limit}
b_{\epsilon}(w,w)
=
2\rmi\lambda\|a\|_{L^2(N)}^2+O(\epsilon)
\qquad \epsilon\downarrow0 .
\end{equation}
Letting $\epsilon\downarrow0$ in \eqref{eq:green-flux} gives $a=0$.  Hence $w$ has no radiating leading term. The remainder is therefore in $L^2$, and elliptic regularity places it in $\operatorname{dom}(\Delta_p)$.  Proposition~\ref{prop:no-positive-eigenforms}, applied with eigenvalue $\lambda^2>0$, now implies $w\equiv0$. Therefore, we have that
\begin{equation}\label{eq:jump-is-poisson}
(R_\lambda-R_{-\lambda})f=E_\lambda(\Psi).
\end{equation}

It remains to compute the coefficient $\Psi$.  Pair the equation
$(\Delta_p-\lambda^2)R_{-\lambda}f=f$ with $E_\lambda(\Phi_\rho)$ and integrate by parts over $M_\epsilon$.  Since $(\Delta_p-\lambda^2)E_\lambda(\Phi_\rho)=0$, the interior terms cancel and only the boundary pairing at $x=\epsilon$ remains.  Using the incoming/outgoing asymptotics of $R_{-\lambda}f$ and $E_\lambda(\Phi_\rho)$ gives
\begin{equation}\label{eq:coefficient-identity}
\langle f,E_\lambda(\Phi_\rho)\rangle_{L^2(M)}
=
-2\rmi\lambda\,\langle \Psi,\Phi_\rho\rangle_{L^2(N)} .
\end{equation}
Since $(\Phi_\rho)_\rho$ is an orthonormal basis of $L^2(N;\Lambda^pT^*N)$, we obtain
\begin{equation}\label{eq:Psi-expansion}
\Psi
=
\frac{\rmi}{2\lambda}
\sum_\rho
\langle f,E_\lambda(\Phi_\rho)\rangle_{L^2(M)}\Phi_\rho .
\end{equation}
The sum converges in $C^\infty(N;\Lambda^pT^*N)$ because $\Psi$ is smooth and the $\Phi_\rho$ form the smooth spectral basis from Lemma \ref{basis}.  Applying the continuous map
$\Phi\mapsto E_\lambda(\Phi)$ gives convergence locally in $C^\infty(M;\Lambda^pT^*M)$ using the bound \eqref{grossboundsJ} and Lemma~\ref{scatteringAbounds} (the bounds \eqref{Abound1}--\eqref{Abound3}), and \eqref{eq:stone-resolvent-jump} follows from \eqref{eq:jump-is-poisson} and \eqref{eq:Psi-expansion}.

Finally, Stone's formula for the self-adjoint operator $\Delta_{p}$ gives 
\begin{equation*}
dE_{\Delta_p}(\lambda^2)
=
\frac{1}{2\pi\rmi}(R_\lambda-R_{-\lambda})\,\der(\lambda^2).
\end{equation*}
Equivalently, for the spectral measure $dB_\lambda$ of $\Delta_p^{1/2}$,
\begin{equation*}
dB_\lambda
=
dE_{\Delta_p}(\lambda^2)
=
\frac{1}{2\pi\rmi}(R_\lambda-R_{-\lambda})\,\der(\lambda^2)
=
\frac{\lambda}{\pi\rmi}(R_\lambda-R_{-\lambda})\,\der\lambda .
\end{equation*}
Substituting \eqref{eq:stone-resolvent-jump} yields \eqref{eq:stone-spectral-measure}.  The functional calculus identity \eqref{StoneFull} follows by integrating $k(\lambda^2)$ against the absolutely continuous part of $dB_\lambda$ and adding the pure point contribution at zero, namely the orthogonal projection onto $\ker\Delta_{p}$. 
\end{proof}

The preceding Stone formula immediately gives the following kernel expansion for rapidly decaying functions of the Laplacian.  We state the result needed below on the scattering manifold.

\begin{theo}[Kernel expansion for rapidly decaying functions, cf. \cite{OS}]\label{stone2}
Let $k:[0,\infty)\to\mathbb C$ be a bounded Borel function such that, for every $q\in\mathbb N$,
\begin{equation}\label{eq:k-rapid-decay}
 |k(w)|\leq C_q(1+w)^{-q},\qquad w\geq0 .
\end{equation}
Then $k(\Delta_{p})|_{\mathrm{ker}\delta}$ is a smoothing operator.  More precisely, it has a smooth Schwartz kernel
\begin{equation*}
\kappa_k\in C^\infty\!\big(M\times M;\Lambda^pT^*M\boxtimes (\Lambda^pT^*M)^*\big),
\end{equation*}
and, for $x,y\in M$,
\begin{equation}\label{eq:kernel-expansion-k}
\kappa_k(x,y)
=
 k(0)\sum_{j=1}^N u_j(x)\otimes u_j^*(y)
+
\frac{1}{2\pi}\sum_\rho\int_0^\infty
 k(\lambda^2)
E_\lambda(\Phi_\rho)(x)\otimes E_\lambda(\Phi_\rho)^*(y)\,\der\lambda .
\end{equation}
The series and integral in \eqref{eq:kernel-expansion-k} converge in
$C^\infty(K\times K;\Lambda^pT^*M\boxtimes (\Lambda^pT^*M)^*)$ for every compact set $K\Subset M$.  Here $u_1,\ldots,u_N$ are the orthonormal basis of $\ker\Delta_{p}$ appearing in Theorem~\ref{stone1}.  
\end{theo}
\begin{proof} The proof follows the argument in \cite{OS} for obstacle scattering in the exterior of a compact set in Euclidean space. 
Using functional calculus we have that $(1+\Delta_p)^{s_1}k(\Delta_p)(1+\Delta_p)^{s_2}$ is bounded as an operator in $L^2(M;\Lambda^pT^*M)$ for all $s_1,s_2\in\mathbb{R}$. Therefore $k(\Delta_p)$ continuously maps $H^s_{\comp}(M;\Lambda^pT^*M)$ to $H^{s+q}_{\loc}(M;\Lambda^pT^*M)$ for all $s\in \mathbb{R}^+_0$ and $q\in\mathbb{R}^+_0$ and has a smooth integral kernel, $\kappa$. Denote by $\kappa_m(\Delta_p)$ the approximation of $k(\Delta_p)$ which is defined by truncating the infinite sum (we write $m$ for the truncation parameter to avoid conflict with the dimension $n$). Since $k$ is Borel it suffices to consider the positive and negative parts separately. Thus 
\begin{align}
\langle P_{ac}\kappa_m(\Delta_p)f,g\rangle =\frac{1}{2\pi}\sum\limits_{\rho, \ell_p(\sigma_\rho)\leq m}\int\limits_0^{\infty}k(\lambda^2)\langle f,E_{\lambda}(\Phi_{\rho})\rangle \langle E_{\lambda}(\Phi_{\rho}),g\rangle \,d\lambda.
\end{align} 
It suffices to establish the statement for positive $k$. Therefore, in the operator sense
\begin{align}
0\leq(1+ \Delta_{p})^{s_1}\kappa_m(\Delta_{p})(1+\Delta_{p})^{s_2}\leq(1+ \Delta_{p})^{s_1}\kappa(\Delta_{p})(1+\Delta_{p})^{s_2}.
\end{align}
For any $\chi_1,\chi_2\in C_0^{\infty}(M)$ we then obtain the estimate 
\begin{align}
&|\langle \kappa_m(\Delta_{p})(\chi_1v),\chi_2w\rangle|\leq |\langle k(\Delta_{p})(\chi_1v),\chi_1v\rangle|^{\frac{1}{2}} |\langle k(\Delta_{p})(\chi_2w),\chi_2w\rangle|^{\frac{1}{2}}\nonumber \leq \\& C_s\|\chi_1v\|_{H^{-s}}\|\chi_2w\|_{H^{-s}}.
\end{align}
Therefore $\kappa_m(\Delta_{p})$ has a smooth integral kernel. Since $\Delta_p E_{\lambda}(\Phi_\rho)=\lambda^2 E_{\lambda}(\Phi_\rho)$, applying $(1+\Delta_p)^r$ in each of the two kernel variables replaces $k$ by $(1+w)^{2r}k$, which again satisfies \eqref{eq:k-rapid-decay}; running the estimate above with this symbol shows that the same bounds hold, uniformly in $m$, for every $C^r(K\times K)$-seminorm. Hence the sequence $\kappa_{m}$ is bounded and equicontinuous in $C^{\infty}(K\times K; \Lambda^pT^*M\boxtimes(\Lambda^pT^*M)^*)$ on every compact $K\Subset M$, and by the Arzel\`a--Ascoli theorem it has a subsequence converging in $C^{\infty}(K\times K)$. Because the set $K$ is compact and the limit is unique, the full sequence $\kappa_{m}$ converges to $\kappa$ in $C^{\infty}(K\times K; \Lambda^pT^*M\boxtimes(\Lambda^pT^*M)^*)$.
\end{proof} 

\section{Discussion}
The explicit phase factors distinguish closed and co-closed tangential channels, a feature not present in the scalar scattering problem. The generalized-eigenform construction also provides a Stone-type spectral representation and a functional calculus in the co-closed sector. In dimension $n=3$ and degree $p=1$, this sector corresponds to divergence-free vector fields and gives the electric-field scattering representation for Maxwell's equations. The comparison with the exact conic model yields a smoothing remainder under infinite-order agreement of the metric at infinity, or the stated Sobolev regularity under finite-order agreement.

\section*{Declarations}
\noindent\textbf{Corresponding author.} Nelia Charalambous (\texttt{nelia@ucy.ac.cy}).\par
\noindent\textbf{Funding.} No external funding was received for this research.\par
\noindent\textbf{Ethics approval.} Not applicable; this is a mathematical research study involving no human participants, human data, or animals.\par
\noindent\textbf{Competing interests.} The authors declare that they have no competing interests.\par

\subsection*{Acknowledgements:} The authors wish to thank Alexander Strohmaier for pointing us in the direction of this problem and for useful discussions. They would also like to thank the organizers and participants of the workshop on Analytic and Geometric Aspects of Spectral Theory at BIRS-CMO. 

\section{Appendix}

\subsection{Incoming and outgoing waves}
Let $f\in C^{\infty}(M;\Lambda^pT^*M)$ with $\lambda\in \mathbb{R}\setminus \{0\}$, and suppose $(\Delta_M-\lambda^2)f$ is compactly supported.  We say $f$ is \textit{outgoing} for $\lambda$ if $f=R_{\lambda}h$ for some compactly supported $h$, and \textit{incoming} if it is outgoing for $-\lambda$.  Since the condition depends only on the behaviour of $f$ at infinity, $f$ is outgoing on $M$ if and only if $f|_{M\setminus K}$ is outgoing.  An outgoing form has the asymptotic expansion
\begin{align}
\Pi_pf=e^{i\frac{\lambda}{x}}\Phi\,(1+O(x)) \qquad x\rightarrow 0,
\end{align}
where $\Phi\in C^{\infty}(N;\Lambda^pT^*N)$. It follows that this expansion can be differentiated term by term.

\subsection{Bessel functions}
For completeness, we recall that the Bessel function $J_{\nu}(z)$ satisfies
\begin{align}\label{besseleq}
z^{2}\frac{{\mathrm{d}}^{2}w}{{\mathrm{d}z}^{2}}+z\frac{\mathrm{d}w}{\mathrm{d%
}z}+(z^{2}-\nu^{2})w=0;
\end{align}
see \cite{olverbook} for a comprehensive treatment. The differentiation formulas are
\begin{align}\label{derivativebessel}
\frac{d}{dz}J_{\nu}(z)=J_{\nu-1}(z)-\frac{\nu}{z}J_{\nu}(z),
\end{align}
or 
\begin{align}\label{derivativebesselpositive}
\frac{d}{dz}J_{\nu}(z)=-J_{\nu+1}(z)+\frac{\nu}{z}J_{\nu}(z),
\end{align}
\begin{align}\label{derivativebessel2}
\frac{d^2}{dz^2}J_{\nu}(z)=J_{\nu-2}(z)-\frac{2\nu-1}{z}J_{\nu-1}(z)+\frac{\nu+\nu^2}{z^2}J_{\nu}(z),
\end{align}
\begin{align}\label{derivativebessel3}
\frac{d^3}{dz^3}J_{\nu}(z)=J_{\nu-3}(z)-\frac{3\nu-3}{z}J_{\nu-2}(z)+\frac{3\nu^2}{z^2}J_{\nu-1}(z) -\frac{2 \nu+ 3\nu^2+\nu^3}{z^3}J_{\nu}(z),
\end{align}
and
\begin{align}
J_{\nu-1}\left(z\right)+J_{\nu+1}\left(z\right)=(2\nu/z)%
J_{\nu}\left(z\right).
\end{align}
The same relations hold for $H^{(1)}_{\nu}(z)$ and $H^{(2)}_{\nu}(z)$. The Bessel and Hankel functions are related by
\begin{align} \label{hankel}
2J_{\nu}(z)=H^{(1)}_{\nu}(z)+H^{(2)}_{\nu}(z).
\end{align}
Solutions to \eqref{besseleq} are collectively called cylinder functions and denoted $\mathcal{C}(z)$. The large-argument expansion uses the coefficients
\[a_{k}(\nu)=\frac{(4\nu^{2}-1^{2})(4\nu^{2}-3^{2})\!\cdots\!(4\nu^{2}-(2k-1)^{2})}{%
k!\,8^{k}},\]
where $a_0(\nu)=1$. From the NIST Digital Library of Mathematical Functions \cite{NIST}, eqs.~10.17.13--15, we have that
\begin{align}\label{eqn:Hankelerror}
{H^{1}_{\nu}}\left(z\right)=\left(\frac{2}{\pi z}\right)^{\frac{1}{2}}e^{%
i\omega}\left(\sum_{k=0}^{\ell-1}(i)^{k}\frac{a_{k}(\nu)}{z^{k}}+R_{%
\ell}^{+}(\nu,z)\right)
\end{align}
where $\ell \in \N$, $\omega=z-\tfrac{1}{2}\nu\pi-\tfrac{1}{4}\pi$ and
\[\left|R_{\ell}^{+}(\nu,z)\right|\leq 2|a_{\ell}(\nu)|\mathcal{V}_{z,\pm i%
\infty}\left(t^{-\ell}\right)\*\exp\left(|\nu^{2}-\tfrac{1}{4}|\mathcal{V}_{z,%
\pm i\infty}\left(t^{-1}\right)\right),\]
where $\mathcal{V}_{z,i\infty}\left(t^{-\ell}\right)$ may be estimated in various sectors as follows
\[\mathcal{V}_{z,i\infty}\left(t^{-\ell}\right)\leq\begin{cases}|z|^{-\ell},&0%
\leq\operatorname{ph}z\leq\pi,\\
\chi(\ell)|z|^{-\ell},&\parbox[t]{224.037pt}{$-\tfrac{1}{2}\pi\leq%
\operatorname{ph}z\leq 0$ or
$\pi\leq\operatorname{ph}z\leq\tfrac{3}{2}\pi$,}\\
2\chi(\ell)|\Im z|^{-\ell},&\parbox[t]{224.037pt}{$-\pi<\operatorname{ph}z\leq%
-\tfrac{1}{2}\pi$ or
$\tfrac{3}{2}\pi\leq\operatorname{ph}z<2\pi$.}\end{cases}\]
Here, $\chi(\ell)$ is defined by
\begin{align*}
 \chi(x) := \pi^{1/2} \Gamma\left(\tfrac{1}{2}x+1\right)/\Gamma \left(\tfrac{1}{2}x+\tfrac{1}{2}\right).
\end{align*}
For $z \to 0$, we have the estimate \cite{NIST} equations 10.7.7 and 10.7.2
\begin{align}\label{eqn:HankelSmallAsympt}
 {H^{(1)}_{\nu}}\left(z\right) &\sim
 \begin{cases}
  -(i/\pi) \Gamma\left(\nu\right)(\tfrac{1}{2}z)^{-\nu} & \text { for } \nu > 0 \\
  (2i/\pi)\log z & \text{ for } \nu = 0
 \end{cases}
\end{align}
Here $\sim$ means that the ratio of left- to right-hand side tends to $1$ as $z \to 0$, and $\log$ denotes the principal branch of the complex logarithm.

From the NIST Digital Library of Mathematical Functions \cite{NIST}, eq.~10.14.4,
\begin{align} \label{besselineq}
|J_{\nu}(z)|\leq \frac{|\frac{z}{2}|^{\nu}e^{|\Im z|}}{\Gamma(\nu+1)}\quad \text{for} \quad \nu>-\frac{1}{2}
\end{align}

The parametric Bessel equation, obtained by introducing three complex parameters $\sigma, p, q$, takes the form
\begin{align}
z^2\frac{d^2}{dz^2}w+(1-2p)z\frac{d}{dz}w+(\sigma^2q^2z^{2q}+p^2-\nu^2q^2)w=0
\end{align}
Let $\mathcal{C}_{\nu}$ and $\mathcal{D}_{\nu}$ denote any cylinder functions ($J_{\nu}$, $H^{(1)}_{\nu}$, or $H^{(2)}_{\nu}$). The particular solution is
\begin{align}\label{parambess}
w(z)=z^p\mathcal{C}_{\nu}(\sigma z^q),
\end{align}
cf.\ \cite{NIST}, eq.~(10.13.5). Setting $\sigma=1$, $p=0$, $q=1$ recovers \eqref{besseleq}.

\end{document}